\documentclass{cmslatex}
\usepackage[paperwidth=7in, paperheight=10in, margin=.875in]{geometry}
 \usepackage[backref,colorlinks,linkcolor=red,anchorcolor=green,citecolor=blue]{hyperref}
\usepackage{amsfonts,amssymb}
\usepackage{amsmath}
\usepackage{graphicx}
\usepackage{cite}
\usepackage{enumerate}
 \usepackage{cases}

\renewcommand{\theequation}{\arabic{section}.\arabic{equation}}

\newcommand{\R}{\mathbb{R}}

\newcommand{\ep}{\varepsilon}

\newcommand{\p}{\partial}
\newcommand{\be}{\begin{equation}}
\newcommand{\ee}{\end{equation}}
\newcommand{\ba}{\begin{aligned}}
	\newcommand{\ea}{\end{aligned}}

\newcommand{\N}{\mathbb N}
\newcommand{\cLep}{\mathcal L_\epsilon}
\newcommand{\fu}{\mathfrak{u}}
\newcommand{\fv}{\mathfrak{v}}
\newcommand{\fp}{\mathfrak{p}}

\newcommand{\cX}{\mathcal X}
\newcommand{\ue}{u_\ep}
\newcommand{\ve}{v_\ep}
\newcommand{\we}{w_\ep}

\newcommand{\fvapp}{\fv_{\text{app}}}

   \allowdisplaybreaks
\begin{document}
	\bibliographystyle{acm}
 \title{Existence and stability of partially congested propagation fronts in a one-dimensional Navier-Stokes model\thanks{Received date, and accepted date (The correct dates will be entered by the editor).}}

          %For each author, make a block with the following macros:

\author{Anne-Laure Dalibard\thanks{Sorbonne Universit\'e, Universit\'e Paris-Diderot SPC, CNRS,  Laboratoire Jacques-Louis Lions, LJLL, F-75005 Paris (dalibard@ann.jussieu.fr)} 
	\and Charlotte Perrin\thanks{Aix Marseille Universit\'e, CNRS, Centrale Marseille, I2M, Marseille, France; (charlotte.perrin@univ-amu.fr).}}

         \pagestyle{myheadings} \markboth{PARTIALLY CONGESTED PROPAGATION FRONTS }{ANNE-LAURE DALIBARD AND CHARLOTTE PERRIN}
         	
         	 \maketitle

          \begin{abstract}
               In this paper, we analyze the behavior of viscous shock profiles  of one-dimensional compressible Navier-Stokes equations with a singular pressure law which encodes the effects of congestion. As the intensity of the singular pressure tends to 0, we show the convergence of these profiles towards free-congested traveling front solutions of a two-phase compressible-incompressible Navier-Stokes system and we provide a refined description of the profiles in the vicinity of the transition between the free domain and the congested domain. In the second part of the paper, we prove that the profiles are asymptotically nonlinearly stable under  small perturbations with zero integral, and we quantify the size of the admissible perturbations in terms of the intensity of the singular pressure.
          \end{abstract}


\section{Introduction}

This paper is concerned with the analysis of viscous shock waves for the following compressible Navier-Stokes system written in Lagrangian mass coordinates $(t,x) \in \R_+ \times \R$ 
(we refer to \cite[Section 1.2]{serre1999}  for details concerning the passage from Eulerian coordinates to Lagrangian mass coordinates)
\begin{subnumcases}{\label{eq:NS-ep}}
\partial_t v - \partial_x u = 0, \\
\partial_t u + \partial_x p_\ep(v) - \mu \partial_x\left(\dfrac{1}{v}\partial_x u \right) = 0,
\end{subnumcases}
where $v$ is the specific volume (the inverse of the density), $u$ is the velocity, $\mu$ is a viscosity coefficient and $p_\ep$ is the pressure.
This latter is assumed to be singular close to the critical value $v^* = 1$,
\begin{equation}\label{eq:p-ep}
p_\ep(v) = \dfrac{\ep}{(v-1)^\gamma} \quad \gamma > 0,
\end{equation}
with $\ep\ll 1$.
We supplement system~\eqref{eq:NS-ep} with initial data
\[
(v,u)(0,\cdot) = (v_0,u_0)(\cdot),
\]
and far-field condition
\begin{equation}\label{far-field}
(v,u)(t,x) \underset{x \rightarrow \pm \infty}{\longrightarrow} (v_\pm,u_\pm).
\end{equation}

\medskip
System~\eqref{eq:NS-ep} was introduced in~\cite{bresch2014} (and~\cite{degond2011} for the inviscid case $\mu =0$) in the context of congested flows, that is in the modeling of flows satisfying the maximal density constraint $\rho = \frac{1}{v} \leq 1$.
Equations~\eqref{eq:NS-ep}-\eqref{far-field} represent an approximation of the following free-congested Navier-Stokes equations
\begin{subnumcases}{\label{eq:NS-0}}
\partial_t v - \partial_x u = 0, \\
\partial_t u + \partial_x p - \mu \partial_x\left(\dfrac{1}{v}\partial_x u \right) = 0, \\
v \geq 1, \quad (v-1) p = 0, \quad p \geq 0,
\end{subnumcases}
with the far-field condition
\[
(v,u,p)(t,x) \underset{x \rightarrow \pm \infty}{\longrightarrow} (v_\pm,u_\pm,p_\pm).
\]
System \eqref{eq:NS-0} consists of a free boundary problem between a free phase $\{v > 1\}$ satisfying compressible pressureless dynamics, and a congested incompressible phase $\{v=1\}$.
The pressure $p$ which is activated in the congested domain can be seen as the Lagrange multiplier associated with the incompressibility constraint $\partial_x u = 0$ satisfied in the congested domain.
Precisely, the study~\cite{bresch2014} (extended to the multi-dimensional case in~\cite{perrin2015}) shows that from a sequence of global strong solutions $(v_\ep,u_\ep, p_\ep(v_\ep))_\ep$ to~\eqref{eq:NS-ep} (cast on $\R_+ \times (0,M)$), one can extract a subsequence %denoted $(v_\ep,u_\ep,p_\ep(v_\ep))$ 
converging weakly  as $\ep \rightarrow 0$ to a global weak solution $(v,u,p)$ of~\eqref{eq:NS-0}.
Note that this convergence result does not imply the existence of solutions which couple effectively both compressible and incompressible dynamics.
In other words, it is not excluded that the solutions of \eqref{eq:NS-0} obtained as limits of those of \eqref{eq:NS-ep} all satisfy $p\equiv 0$ or $v\equiv 1$.
Note also, that the present problem is quite different from ``classical'' free boundary problems between two immiscible compressible and incompressible phases studied for instance in  \cite{denisova2018,shibata2016,colombo2016}.
Indeed, the interface between the compressible and the incompressible domains for the congestion problem is not closed since there are mass exchanges between the free and the congested phases. 
This considerably complicates the analysis of the equations.
To the knowledge of the authors, nothing seems to be known concerning the local well-posedness of the general free-congested Navier-Stokes equations \eqref{eq:NS-0}, except the recent results of Lannes et al. \cite{iguchi2018, bresch2019} concerning the one-dimensional floating body problem which can be viewed as a particular inviscid congestion problem.

Although the rigorous justification of singular limit $\ep \rightarrow 0$ is, to the knowledge of the authors, an open problem in the inviscid case $\mu=0$
(in the case of two immiscible fluids a similar singular limit has been studied in \cite{colombo2016,guerra2016} but, as explained before, the congestion problem in the present paper is rather different since the phases cannot be considered here as immiscible), the formal link between models~\eqref{eq:NS-ep} and~\eqref{eq:NS-0} has been used from the numerical point of view in~\cite{bresch2017,degond2011} to investigate the transition at the interface between the congested domain and the free domain. 
The study of Bresch and Renardy~\cite{bresch2017} provides numerical evidence of apparition of shocks on $v$ and $u$ at the interface when a congested domain is created in the system. 
The paper of Degond et al.~\cite{degond2011} contains an analysis of the asymptotic behavior of approximate solutions $(v_\ep,u_\ep)$ of the inviscid Riemann problem associated with the initial data $(v_\ep,u_\ep)(0,\cdot) = (v_-^\ep,u_-)\mathbf{1}_{\{x < 0\}} + (v_+,u_+)\mathbf{1}_{\{x > 0\}}$ where $v_-^\ep \rightarrow 1$ and $v_+ > v_-^\ep$ remains far from $1$.
Both studies present free-congested solutions for the compressible-incompressible Euler equations obtained from the singular compressible Euler equations~\eqref{eq:NS-ep} ($\mu=0$) via the formal limit $\ep \rightarrow 0$.
Up to our knowledge, nothing seems to be known regarding the stability of such congestion fronts.
Furthermore, no explicit free-congested solution to~\eqref{eq:NS-0} for $\mu >0$ has been exhibited so far. 

\bigskip
The goal of this paper is two-fold. On the one hand, we study the asymptotic behavior of traveling wave solutions of \eqref{eq:NS-ep} connecting an almost congested left state $v_-^\ep = 1 + \ep^{1/\gamma}$, to a non-congested right state $v_+ > 1$.
On the other hand, we prove the non-linear asymptotic stability of such profiles uniformly with respect to the parameter~$\ep$. 

\bigskip
The first result of stability of traveling waves for the standard compressible Navier-Stokes equations
\begin{subnumcases}{\label{eq:NS-classic}}
\partial_t v - \partial_x u = 0 \\
\partial_t u + \partial_x P(v) - \partial_x\left(\dfrac{\mu}{v}\partial_x u\right) = 0
\end{subnumcases}
with the pressure $P(v) = \frac{a}{v^\gamma}$, $\gamma \geq 1$ and $a> 0$, was obtained by Matsumura and Nishihara in~\cite{matsumura1985}.
Matsumura and Nishihara showed that there exists a unique (up to a shift) traveling wave $(v,u)(t,x) = (\fv,\fu)(x-st)$ connecting the two limit states $(v_\pm,u_\pm)$ at $\pm \infty$, provided that $0 < v_- < v_+$ and $u_+ < u_-$ where $v_\pm,u_\pm$ are related to the shock speed $s$ through the Rankine-Hugoniot conditions (see~\eqref{eq:TW-1} below). 
Under some restriction on the amplitude of the shock $|p(v_+)-p(v_-)|\leq C(v_-,\gamma)$, they established next the asymptotic stability of $(\fv,\fu)$ with respect to small initial perturbations $(v_0-\fv,u_0-\fu) \in H^1(\R) \cap L^1(\R)$ with zero integral, {\it i.e.} perturbations for which there exists $(V_0,U_0)\in H^2(\R)$ with
\[
v_0 - \fv = \partial_x V_0\in L^1_0(\R), \quad u_0 - \fu = \partial_x U_0\in L^1_0(\R).
\] 
The restriction on the amplitude of the shock amounts to assuming that $(\gamma-1)\times$(total variation of the initial data) is small. In particular for $\gamma = 1$ there is no restriction on the amplitude of the shock.
The result is achieved by means of suitable weighted energy estimates on the integrated quantities $V$ and $U$.

Later on, several works generalized this result by considering non-zero mass perturbations and shocks with larger amplitude~\cite{mascia2004,liu2009,humpherys2010}.
Besides, the numerical study carried out in~\cite{humpherys2010} seems to indicate that the profiles should be stable independently of the shock amplitude.
In the case of viscosities depending in a non-linear manner on $1/v$, {\it i.e.} $\mu(v) = \mu v^{-(\alpha +1)}$, Matsumura and Wang~\cite{matsumura2010} managed to adapt the weighted energy method for suitable parameters $\alpha$. 
Without any smallness assumption on the amplitude of the shock, they proved the non-linear asymptotic stability for perturbations with zero mass provided that $\alpha \geq \frac{1}{2}(\gamma-1)$.\\
The constraint on the parameter $\alpha$ was finally removed in the recent paper of Vasseur and Yao~\cite{vasseur2016}.
The originality of their method consists in rewriting the system~\eqref{eq:NS-classic} with the new velocity (also called \emph{effective velocity}) $w = u - \frac{\mu}{\alpha} \partial_x v^{-\alpha}$ if $\alpha \neq 0$ and $w = u - \mu\partial_x \ln v$ if $\alpha =0$:
\begin{subnumcases}{\label{eq:NS-vw-classic}}
\partial_t v - \partial_x w - \partial_x\left(\dfrac{\mu}{v^{\alpha +1}}\partial_x v\right) = 0, \\
\partial_t w + \partial_x P(v)  = 0,
\end{subnumcases}
where the specific volume $v$ satisfies now a parabolic equation. 
The regularization effect on $v$ induced by this change of unknown was previously identified by Shelukhin~\cite{shelukhin1984} in the case $\alpha = 0$ and by Bresch, Desjardins~\cite{bresch2003,bresch2006}, Mellet, Vasseur \cite{mellet2008}, Haspot~\cite{haspot2014,haspot2018} for more general viscosity laws. 
It enables the derivation of an entropy estimate (also called \emph{BD entropy estimate}) in addition to the classical energy estimate.
In the non-linear stability study of Vasseur and Yao, the introduction of the effective velocity helps for the treatment of the non-linear terms (see $F$ and $G$ in~\eqref{eq:NL-VW-lin} below) and consequently it allows to consider any coefficient $\alpha \in \R$ which was not the case in~\cite{matsumura2010}.\\
We show in this paper that the new formulation in $(v,w)$ turns out to be also interesting when considering singular pressure laws like~\eqref{eq:p-ep}.
Although our study is restricted to linear viscosity coefficients ($\alpha=0$), which corresponds to the case initially treated in~\cite{bresch2014}, 
we could a priori extend our result to viscosities $\dfrac{\mu}{v^{\alpha +1}}$ like in~\cite{vasseur2016} without any substantial difficulty.

\subsection*{Main results}

Our first result concerns the existence and qualitative asymptotic behavior of solutions of \eqref{eq:NS-ep}-\eqref{far-field}:

\medskip
\begin{prop}[Description of partially congested profiles] 
	Assume that the pressure law is given by \eqref{eq:p-ep}.
	\begin{enumerate}
		\item Let $1<v_-<v_+$, and let $u_+, u_-$ such that 
		\[
		(u_+-u_-)^2 =- (v_+-v_-)(p_\ep (v_+) - p_\ep(v_-)).
		\]
		Then there exists a unique (up to a shift) traveling front solution of \eqref{eq:NS-ep}-\eqref{far-field} $(u,v)(t,x)=(\fu_\ep,\fv_\ep)(x-s_\ep t)$. The shock speed $s_\ep$ satisfies the Rankine-Hugoniot condition
		\[
		s_\ep^2=-\frac{p_\ep(v_+)-p_\ep(v_-)}{v_+-v_-}.
		\]
		
		\item Take $v_-=1+ \ep^{1/\gamma}$, $v_+>1$ (independent of $\ep$). Let 
		\[
		r:=\frac{v_+}{\mu\sqrt{v_+-1}}
		\]
		and define the partially congested profile $(\bar{\fu}, \bar \fv)$ such that
		\[
		\bar \fv(\xi):=\begin{cases}
		1 &\text{ if }\xi<0\\
		\dfrac{v_+}{1+(v_+-1)e^{-r\xi}} &\text{ if }\xi\geq 0
		\end{cases},
		\qquad \bar \fu' = - \dfrac{u_- - u_+}{v_+-1} \ \bar\fv', 
		\]
		which is solution to the limit system \eqref{eq:NS-0}.
		
		Then 
		\be\label{limit-profile}
		\lim_{\ep\to 0} \sup_{\xi \in \R} \inf_{C\in \R} | \fv_\ep(\xi+C) - \bar \fv(\xi)|=0.
		\ee
		
		\item Assume additionally that $\gamma \geq 1$ and
		fix the shift in $\fv_\ep$ by choosing $\fv_\ep(0)$ such that $\fv_\ep(0)-1\propto \ep^{\frac{1}{\gamma+1}}$. There exist constants $\bar C, \underline C, \bar \sigma, \underline \sigma$, independent of $\ep$, and a number $\xi_\ep$ such that $\lim_{\ep \to 0} \xi_\ep=0$, such that for all $\xi<\xi_\ep$,
		\begin{equation}\label{eq:control-cong}
		\underline C \ep^{1/\gamma} \exp( \underline\sigma \ep^{-1/\gamma}\xi ) \leq \fv_\ep(\xi )- v_- \leq \bar C \ep^{1/\gamma}\exp( \bar\sigma \ep^{-1/\gamma}\xi ).
		\end{equation}
	\end{enumerate}
	\label{prop:profile}
\end{prop}

\medskip
\begin{rem}
	
	\medskip
	\begin{itemize}
		\item We recall that $\fv_\ep$ is defined up to a shift. 
		Taking the infimum over the parameter $C$ in \eqref{limit-profile} amounts to fixing this shift.
		
		\medskip
		\item Note that the limit profile $\bar \fv$ is also the specific volume profile for the traveling wave solution of \eqref{eq:NS-0}.
		
		\medskip
		\item In the last item of the proposition we impose the value $\fv_\ep(0)$, which amounts to  prescribing the  shift $C$.
		Thanks to the specific scaling that we have chosen, we will see that $p_\ep(\fv_\ep)$ converges towards zero  uniformly in $[0, + \infty[$, and that $\fv_\ep \to 1$ in $]-\infty, 0]$.
		This means that in the limit the zone $\xi<0$ corresponds to the congested zone, in which $\bar \fv=1$, while the zone $\xi>0$ is the free zone.
		Fixing $\fv_\ep(0)$ also enables us to get the explicit control \eqref{eq:control-cong} of the distance between $\fv_\ep$ and the end state $v_-$ in the congested zone.
		
		\medskip
		\item The end state of the congested zone, $v_- = v_-^\ep = 1 + \ep ^{1/\gamma}$, is chosen so that $p_\ep(v_-)=1$ for all $\ep> 0$.
		Of course, any choice of sequence $(v_-^\ep)_{\ep> 0}$ such that $\lim_{\ep \to 0} p_\ep(v_-^\ep) \in \ ]0, +\infty[$ would lead to similar results.
		We refer to Remark \ref{rem:ext-tv} below for details.
	\end{itemize}
\end{rem}

\medskip

Actually, we are able to give a more refined description of the  behavior close to the transition zone $\xi=0$, and to give a quantitative error estimate. We have the following proposition, and we refer to Section 2 for more details:

\medskip
\begin{prop}\label{prop:est-transition}
	Let  $v_-=1+ \ep^{1/\gamma}$ and assume that $\gamma \geq 1$.
	We fix the shift in $\fv_\ep$ by setting $\fv_\ep(0)-1\propto\ep^{\frac{1}{\gamma +1}}$.
	
	\begin{enumerate}
		\item For all $R>0$, there exists a constant $C_R$ such that
		\[
		\| \fv_\ep - \bar \fv\|_{L^\infty(-R,R)} \leq C_R \ep^{\frac{1}{\gamma +1}}.
		\]

		\item Let $\tilde v$ be the solution of the ODE 
		\[
		\tilde v'=(\mu \bar s)^{-1} (1- \tilde v^{-\gamma}), \quad \tilde v(0)=2,
		\]
		and let $ \xi^*<0$ be a suitable parameter such that $\xi^*=O(\ep^{\frac{1}{\gamma+1}})$. Then
		\[
		\left| \fv_\ep(\xi) - \bar \fv (\xi) -\ep^{1/\gamma} \tilde v \left(\frac{\xi-\xi^* }{\ep^{1/\gamma}}\right)\right|\leq C  \ep^{\frac{1}{\gamma +1}}|\xi|  \quad \forall \xi\in [\xi_{min}, 0].
		\]
		where the number $\xi_{min}<0$ is such that $\xi_{min}\sim - C \ep^{\frac{1}{\gamma +1}}$.
	\end{enumerate}
\end{prop}

\medskip
The proofs of  Propositions \ref{prop:profile} and \ref{prop:est-transition} rely on ODE arguments.
Combining the two equations of~\eqref{eq:NS-ep}, we find an ODE satisfied by $\fv_\ep$, for which we prove the existence and uniqueness of solutions. 
Compactness of solutions easily follows from the bounds on $\fv_\ep$, and therefore on its derivative (using the equation), and we pass to the limit in the ODE in order to find the limit equation satisfied by $\bar \fv$. 
We then use barrier functions to control the behavior of $\fv_\ep$ in the congested zone ($\xi\to -\infty$), and energy estimates (in this case, a simple Gronwall lemma) to control the error between $\fv_\ep$ and $\fvapp$ in the transition zone. 

\bigskip
The second part of this paper is devoted to the analysis of the stability of the profiles $(\ue,\ve):=(\fu_\epsilon,\fv_\ep )(x-s_\ep t)$ in the regimes where $\ep$ is very small.
To that end, we follow the overall strategy of~\cite{vasseur2016} and introduce the effective velocity $w = u - \mu\partial_x\ln v$.
Equations~\eqref{eq:NS-ep} rewrite in the new unknowns $(w,v)$
\be
\label{eq-NL-vw}
\ba
\p_t w + \p_x p_\ep(v)=0,\\
\p_t v - \p_x w - \mu \p_{xx} \ln v=0.
\ea
\ee
The profile $(\we = \ue - \mu\partial_x \ln \ve, \ve)$ is then a solution of \eqref{eq-NL-vw}. \\
The second ingredient that we need for the derivation of suitable energy estimates is the passage to the integrated quantities.
Consider an initial data $(w_0, v_0)\in ((\we)_{|t=0}, (\ve)_{|t=0}) + L^1_0\cap L^\infty(\R)^2$, where $L^1_0(\R)$ is the set of $L^1$ functions of zero mass. 
We can then introduce $(W_0,V_0)$ such that
\begin{equation}\label{eq:WV-0}
W_0(x)=\int_{-\infty}^x(w_0(z)-\we(z)) \ dz , 
\quad V_0(x)=\int_{-\infty}^x(v_0(z)-\ve(z))\ dz.
\end{equation}
Assuming that this property remains true for all time, that is $(w-\we, v-\ve)(t)\in L^1_0(\R)$ $\forall t\geq 0$, we define
\[
W(t,x)=\int_{-\infty}^x(w(t,z)-\we(t,z))\ dz,\quad V(t,x)=\int_{-\infty}^x(v(t,z)-\ve(t,z))\ dz.
\]
Then $(W,V)(t,x)\to 0$ as $|x|\to \infty$, and $(W,V)$ is a solution of the system
\be\label{eq:NL-VW}
\ba
\p_t W + p_\ep (\ve + \p_x V) - p_\ep(\ve)=0,\\
\p_t V - \p_x W - \mu \p_x \ln \frac{\ve + \p_x V}{\ve}=0, \\
(W,V)_{|t=0} = (W_0,V_0).
\ea
\ee
In the rest of the paper, we shall assume that $\ep < \ep_0$ for a constant $\ep_0$ small enough (depending only on $v_+,\mu, \gamma$).

\medskip
\begin{thm}[Existence of a global strong solution $(W,V)$]{\label{thm:WV}}
	Assume that $(W_0,V_0) \in (H^2(\R))^2$ with
	\begin{equation}\label{eq:init-1}
	\sum_{k=0}^2 \ep^{\frac{2k}{\gamma}}\int_{\R}{\left[ \dfrac{|\partial^k_x W_0|^2}{-p'_\ep(\ve)} + |\partial^k_x V_0|^2 \right] dx} \leq \delta_0^2 \ep^{\frac{5}{\gamma}}
	\end{equation}
	for some $\delta_0$ small enough, depending only on $v_+$, $\gamma$ and $\mu$.
	Then there exists a unique global solution $(W,V)$ to~\eqref{eq:NL-VW} satisfying
	\begin{align*}
	& W \in \mathcal{C}([0;+\infty);H^2(\R)), \\
	& V \in \mathcal{C}([0;+\infty);H^2(\R)) \cap L^2(\R_+; H^3(\R)).
	\end{align*}
	Moreover there exists $C > 0$ depending only on $v_+,\mu, \gamma,\delta_0,$ such that
	\begin{equation}\label{estNL}
	\sum_{k=0}^2 \ep^{\frac{2k}{\gamma}} \left[ \sup_{t \geq 0} \int_{\R}{\left( \dfrac{|\partial^k_x W|^2}{-p'_\ep(\ve)} + |\partial^k_x V|^2 \right) dx}
	+  \int_{\R_+}\int_{\R}{ \left( \partial_x \ve |\partial^k_x W|^2 + |\partial^{k+1}_x V|^2 \right) dx \ dt} \right]  
	\leq C \ep^{\frac{5}{\gamma}}.
	\end{equation}
\end{thm}

\medskip
\begin{rem}
	
	\medskip
	\begin{itemize}
		\item The weight $(-p_\ep'(\ve))^{-1}$ is of order $\ep^{1/\gamma}$ in the  congested zone (in which $\ve-1=O(\ep^{1/\gamma})$), and of order $\ep^{-1}$ in the non-congested zone (in which $\ve-1$ is bounded away from zero). 
		Hence the presence of this weight induces an additional loss of control on $W$ in the congested zone. 
		
		\medskip
		\item The control by $C\ep^{\frac{5}{\gamma}}$ with $C$ small enough in \eqref{estNL} ensures in particular the lower bound $v = \ve + \partial_x V > 1$. 
		Indeed,
		\begin{eqnarray}\label{eq:v-Linfty}
		\|\p_x V\|_{L^\infty_x} &\leq& \sqrt{2} \|\p_x V\|_{L^2_x}^{1/2}\|\p_x^2  V\|_{L^2_x}^{1/2}\\\nonumber&
		\leq& \sqrt{2}\left(C^{1/2} \ep^{\frac{5}{2\gamma}- \frac{1}{\gamma}}\right)^{1/2}\left(C^{1/2} \ep^{\frac{5}{2\gamma}- \frac{2}{\gamma}}\right)^{1/2}\\\nonumber&
		\leq&\sqrt{2} C^{1/2} \ep^{1/\gamma}. 
		\end{eqnarray}	
		Hence, if $C<1/2$, we have $\ve + \partial_x V>1$. 
		
	\end{itemize}
\end{rem}

\bigskip
Under the previous assumptions, we show the following stability result on the variable $(u,v)$.

\medskip
\begin{thm}[Nonlinear asymptotic stability of partially congested profiles]\label{thm:estimates}
	Assume that the initial data $(u_0,v_0)$ is such that
	\[
	u_0 - (u_\ep)_{t=0} \in W^{1,1}_0(\R)\cap H^1(\R), \quad
	v_0 - (v_\ep)_{t=0} \in W^{2,1}_0(\R)\cap H^2(\R),
	\]
	and the associated couple $(W_0,V_0) \in H^2\times H^3 (\R)$ satisfies~\eqref{eq:init-1}.
	Then there exists a unique global solution $(u,v)$ to~\eqref{eq:NS-ep} which satisfies
	\begin{align*}
	& u-\ue \in \mathcal{C}([0;+\infty);H^1(\R)\cap L^1_0(\R)), \\
	& v-\ve \in \mathcal{C}([0;+\infty);H^1(\R)\cap L^1_0(\R)) \cap L^2(\R_+; H^2(\R))
	\end{align*}	
	and
	\begin{equation}
	v(t,x) > 1 \quad \text{for all}~ t,x.
	\end{equation}
	More precisely, there exists $C_1 > 0$ only depending on $v_+,\mu, \gamma$ and the initial data, such that
	\[
	\| u - \ue \|_{L^\infty(\R_+; H^1(\R))} + \| v - \ve \|_{L^\infty(\R_+; H^1(\R))} + \| v - \ve \|_{L^2(\R_+; H^2(\R))} \leq C_1
	\]	
	and on any finite time interval $[0,T]$, there exists another positive constant $C_2$ depending additionally on $T$ and $\ep$, such that
	\[
	\| u- \ue \|_{L^\infty(0,T;L^1(\R))} + \| v - \ve \|_{L^\infty(0,T;L^1(\R))} \leq C_2(T,\ep).
	\]
	Finally
	\begin{equation}
	\sup_{x\in \R} \Big| \big((u,v)(t,x) - (\ue,\ve)(t,x)\big)\Big| \underset{t\rightarrow +\infty}{\longrightarrow} 0.
	\end{equation}
\end{thm}

\medskip

\begin{rem}
	Note that the theorem states that $(u-\ue)(t)$ and $(v-\ve)(t)$ are functions of $L^1_0(\R)$ which justifies a posteriori the passage to the integrated system~\eqref{eq:NL-VW}.
\end{rem}

\medskip

\begin{rem}
	If the previous theorem states the stability of the approximate profiles $(v_\ep,u_\ep)$, the stability of the limit profile $(\bar{v}, \bar{u})$ remains open. 
	Indeed, the estimates (in particular \eqref{eq:v-Linfty}) that we derive all degenerate as $\ep \rightarrow 0$ and therefore do not give any information in the limit.
\end{rem}

\medskip
The proofs of Theorem  \ref{thm:WV} and Theorem \ref{thm:estimates} rely on several ingredients. First, we derive weighted $H^2$ estimates for equations \eqref{eq:NL-VW}, using the structure of the linearized system. 
We then obtain $L^1$ bounds by a method similar to the one used by Haspot in \cite{haspot2018}. Eventually, the long-time stability of $(\ue,\ve)$  follows easily.

\bigskip

\begin{rem}
	Note that the assumption $\gamma \geq 1$ is used only in the last point of the Proposition \ref{prop:profile} and \ref{prop:est-transition}. 
	The other results still hold for $\gamma > 0$ and more generally for pressure laws defined on $\ ]1,+ \infty[ \ $ which are singular close to $v=1$ (provided that $v_-$ is well-chosen for the second point), strictly decreasing and convex on $]v_-,v_+[$.
	The convexity of the pressure law on the interval $]v_-,v_+[$ is crucial for the existence of monotone profiles $(\fv_\ep,\fu_\ep)$ joining the states $(v_-,u_-)$ and $(v_+,u_+)$ (cf. Proposition \ref{prop:profile}). 
	The monotonicity of the profiles is then an essential property for the stability results which follow ({\it cf.} Theorem \ref{thm:WV} and \ref{thm:estimates}). 
	The specific form of the pressure \eqref{eq:p-ep} (which blows up close to $1$ like a power law) is used in all the results of this paper to exhibit the small scales associated to the singular limit $\ep \to 0$.
	Nevertheless, we expect similar results for more general (strictly decreasing, convex on $]1,v_+[$, singular at $1$) pressure laws. 
	All the estimates will then depend on the specific balance between the parameter $\ep$ and the type of the singularity close to $v=1$ encoded in the pressure law. 
\end{rem}

\bigskip
The paper is organized as follows. 
Section \ref{sec:profile} is concerned with the description of partially congested solutions of \eqref{eq:NS-0} and the proof of Propositions \ref{prop:profile} and \ref{prop:est-transition}.
Sections~\ref{sec:estimates} and \ref{sec:stability} are devoted to the proof of the stability Theorems \ref{thm:WV} and \ref{thm:estimates}.
Finally, we have postponed to the last section \ref{sec:appendix} the proof of some technical lemmas.

\section{Partially congested profiles}\label{sec:profile}

This section is devoted to the proof of  Propositions \ref{prop:profile} and \ref{prop:est-transition}. 
In the first paragraph, we study the existence and properties of traveling fronts of the limit system \eqref{eq:NS-0}. 
We then investigate the asymptotic behavior of traveling fronts for the system with singular pressure \eqref{eq:NS-ep}. 
Classically, we prove that such traveling fronts solve an ODE, and we compute an asymptotic expansion for solutions of this ODE.

\subsection{Traveling fronts  of \eqref{eq:NS-0}.}
Let $v_-=1 < v_+$, $u_-> u_+$ and $(u,v,p)$ be a solution of \eqref{eq:NS-0} of the form $(\fu,\fv,\fp)(x-st)$ satisfying the far-field condition %\eqref{far-field}.
\[
(v,u,p)(t,x) \underset{x \rightarrow \pm \infty}{\longrightarrow} (v_\pm,u_\pm,p_\pm),
\]
with $p_\pm$ determined below.
We look for a profile $(\fu,\fv,\fp)$ whose congested zone is exactly $(-\infty, \xi^*)$ for some $\xi^*\in \R$ (we will justify this simplification in Remark \ref{rem:congestion-interval} below). 

In the free zone, i.e. in the domain $\{ \fv >1\}$, we have $\fp = 0$ and
\[
\begin{cases}
-s  \fv'(\xi) - \fu'(\xi) = 0 \\
-s  \fu'(\xi) - \mu \left(\dfrac{\fu'}{\fv}\right)'(\xi) = 0
\end{cases}
\quad \forall \ \xi > \xi^*,
\]
which by integration yields
\begin{equation}\label{eq:TW0-free}
\begin{cases}
s  \fv(\xi) + \fu(\xi) = s  v_+ + u_+\\
s \fu(\xi) + \mu\dfrac{\fu'(\xi)}{\fv(\xi)} = s  u_+ 
\end{cases}
\quad \forall \ \xi > \xi^*
\end{equation}
using the fact that $\fu' \rightarrow 0$ as $\xi \rightarrow +\infty$. As a consequence, in the free zone, $\fu$ is a solution of the logistics equation
\be\label{ODE-fu}
\fu'=\frac{1}{\mu}(u_+ - \fu)\left( s v_+ + u_+ - \fu\right).
\ee
Using the relation $-s  \fv'(\xi) = \fu'(\xi)$ and \eqref{eq:TW0-free}, we find that $\fv$ satisfies in the free zone
\[
\fv' = \frac{s }{\mu}\fv(v_+ - \fv).
\]	
Now, in the congested domain we have $\fv = 1$ and
\[
\begin{cases}
\fu'(\xi) = 0\\
-s \fu'(\xi) + \fp'(\xi) - \mu \, \fu''(\xi) = 0
\end{cases}
\quad \forall \xi < \xi^*.
\]
Since $\fu$ is constant in the congested domain, the previous equations are rewritten as
\[
\begin{cases}
\fv(\xi) = v_- = 1\\
\fu(\xi) = u_-\\
\fp(\xi) = \mathrm{cst}=: p_-
\end{cases}
\quad \forall \xi < \xi^*.
\]
We now find the value of $p_-$ by making the following requirements, which ensure that $(u,v,p)$ is a solution of \eqref{eq:NS-0} in the whole domain:
\begin{itemize}
	\item $\fu$ and $\fv$ are continuous at $\xi=\xi^*$;
	\item $ \fp -\mu \dfrac{\fu'}{\fv}$ is continuous at $\xi=\xi^*$. 
\end{itemize}
These conditions lead to the Rankine-Hugoniot condition
\begin{equation}\label{eq:s-limit}
s = \dfrac{u_- - u_+}{v_+-1} > 0,
\end{equation}
and to the initial condition $\fu((\xi^*)^+)= u_-$ for the logistics equation \eqref{ODE-fu}. We infer that
\begin{eqnarray}
p_- &=& -\mu \lim_{\xi \rightarrow (\xi^*)^+}\fu'(\xi)\nonumber\\
&=& s^2(v_+-1)\nonumber\\
&=&\dfrac{(u_- - u_+)^2}{v_+-1}.\label{eq:fp}
\end{eqnarray}

\medskip
\begin{rem}
	The expression of the pressure~\eqref{eq:fp} does not depend on the viscosity $\mu$ and is actually the same as the one obtained by Degond, Hua and Navoret~\cite{degond2011} for the free-congested Euler system ({\it cf.} Case 2 of Proposition 5 in~\cite{degond2011}).
\end{rem}

\medskip
We emphasize that in the limit system, there is no constraint between $u_-, u_+$ and $v_+$ (as long as $p_-$ is free).
Conversely, instead of imposing the far-field condition $u_-$, we could fix the pressure $\fp$ in the congested domain and deduce the corresponding $u_-$ by~\eqref{eq:fp}.

\medskip
\begin{rem}
	Let us now prove that restricting the analysis to profiles whose congested zone is of the form $(-\infty, \xi^*)$ is legitimate. By continuity, the non-congested zone $\{\fv >1\}$ is an open set, and therefore  a countable union of disjoint open intervals. Let $I\subset \R$ be one of these intervals. We argue by contradiction and assume that $I=]a,b[$ with $a,b\in \R$. Then, reasoning as above, we infer that $\fu$ satisfies a logistics equation on the interval $]a,b[$. Furthermore $\fv(a)=\fv(b)=1$ (otherwise $I$ could be extended), and thus $\fu(a)=\fu(b)$. We deduce that $\fu$ is constant on $I$, and as a consequence $\fv$ is also constant - and therefore identically equal to 1 -  on $I$: contradiction.
	Therefore $a=-\infty$ or $b=+\infty$. 
	Since $\fv(-\infty)=1$ and $\fv(+\infty)= v_+>1$, we deduce that $\{\fv >1\}= \ ]\xi^*, + \infty [$ for some $\xi^*\in \R$. 
	\label{rem:congestion-interval}
\end{rem}

\subsection{Existence and uniqueness (up to a shift) of traveling fronts.}
Assume that $(u,v)$ is a solution of \eqref{eq:NS-ep} of the form $(\fu_\ep,\fv_\ep)(x-s_\ep t)$. Plugging this expression into \eqref{eq:NS-ep}, we find
\begin{equation}\label{eq:TW-0}
\begin{cases}
-s_\ep  \fv_\ep'(\xi) - \fu_\ep'(\xi) = 0 \\
-s_\ep  \fu_\ep'(\xi) + \big(p_\ep(\fv_\ep)\big)'(\xi) - \mu \left(\dfrac{1}{\fv_\ep}\fu_\ep'\right)'(\xi) = 0
\end{cases}
\end{equation}
where $\xi := x-s_\ep t$.
We integrate the previous equations over $(\pm \infty, \xi)$ to get
\begin{subnumcases}{\label{eq:TW-1}}
s_\ep  \fv_\ep + \fu_\ep = s_\ep  v_\pm + u_\pm \label{eq:TW-1-1}\\
-s_\ep \fu_\ep + p_\ep(\fv_\ep) - \mu\dfrac{\fu_\ep'}{\fv_\ep} = -s_\ep  u_\pm + p_\ep(v_\pm) \label{eq:TW-1-2}
\end{subnumcases}
using the fact that $\fu_\ep' \rightarrow 0$ as $|\xi| \rightarrow \infty$. This leads to the condition
\[
\frac{u_+-u_-}{v_+-v_-}=- \frac{p_\ep (v_+)-p_\ep (v_-)}{u_+-u_-},
\]
and therefore $(u_+-u_-)^2=-(p_\ep (v_+)-p_\ep (v_-))/(v_+-v_-)$.
The shock speed is then
\begin{equation}
s_\ep  = \pm \sqrt{-\dfrac{p_\ep(v_+) - p_\ep(v_-)}{v_+ - v_-} }.
\end{equation}
If $s_\ep>0$ (resp. $s_\ep<0$), the traveling front is moving to the right (resp. to the left). The ODE satisfied by $\fv_\ep$ follows from
the relation $\fu_\ep' = -s\fv_\ep'$ inserted in~\eqref{eq:TW-1-2} 
\begin{align}\label{eq:ED-TW}
\fv_\ep' 
& = \dfrac{\fv_\ep}{\mu s_\ep} \big( s_\ep^2 (v_+ - \fv_\ep) + p_\ep(v_+) - p_\ep(\fv_\ep)  \big)   \\
& = \dfrac{\fv_\ep}{\mu s_\ep} \big( s_\ep^2 (v_- - \fv_\ep) + p_\ep(v_-) - p_\ep(\fv_\ep)  \big).  \nonumber
\end{align}
Now, assume that $v_- < v_+$, and let $v_0\in ]v_-, v_+[$ be arbitrary, and consider the Cauchy problem \eqref{eq:ED-TW} endowed with the initial data $\fv_\ep(0)=v_0$. It has a unique maximal solution according to the Cauchy-Lipschitz theorem. Since $\fv=v_\pm$ is a constant solution of \eqref{eq:ED-TW}, we infer that $\fv_\ep \in ]v_-, v_+[$, and therefore the solution is global. Since the function $p_\ep$ is convex, it is easily proved that
$s_\ep^2 (v_+ - v) + p_\ep(v_+) - p_\ep(v)>0$ for all $v\in ]v_-, v_+[$. Therefore $\fv_\ep$ is a monotone function. Since we require that $v_-<v_+$, this implies that $\fv_\ep$ is necessarily increasing, and consequently $s_\ep>0$. Hence $\fv_\ep:\R \mapsto ]v_-, v_+[$ is one-to-one and onto. Classically, all other solutions of \eqref{eq:ED-TW} satisfying the far-field conditions \eqref{far-field} are translations of this profile. This proves the first statement of Proposition \ref{prop:profile}.

\begin{figure}[t]
	\begin{center}
		\includegraphics[scale=0.5]{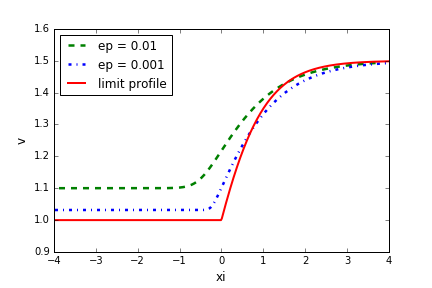}
	\end{center}
	\caption{{\label{fig:profiles}} Asymptotic behavior of the profiles $\fv_\ep$ with $v_-^\ep = 1+ \ep^{1/\gamma}$, $v_+=1.5$ and $\gamma=2$. The shift is fixed by prescribing $\fv_\ep(0) = 1 + \ep^{\frac{1}{\gamma+1}}$.
	}
\end{figure}

\subsection{Qualitative asymptotic description of traveling fronts.}\label{ssec:qualitative-cv}
In the rest of this paper, we are interested in the case when   $v_+>1$ is a fixed number, independent of $\ep$ (the zone on the right is not congested), and $\lim_{\ep\to 0} v_-^\ep=1$ (the zone on the left is asymptotically congested, see Figure~\ref{fig:profiles}). We focus on
traveling fronts such that $s_\ep\to \bar s \in ]0, + \infty[$, or equivalently $\liminf p_\ep(v_-)>0$.  It is easily checked that this implies $v_-^\ep= 1 + C_0 \ep^{1/\gamma} + o(\ep^{1/\gamma})$ for some positive constant $C_0$. This justifies our choice $v_-^\ep=1 + \ep^{1/\gamma}$ which yields $\bar{s}^2 = (v_+-1)^{-1}$. In the sequel, we will abusively write $v_-$ in place of $v_-^\ep$ in order to alleviate the notation.

\medskip

\begin{rem}
	Note that if we choose $v_-= 1 + C_0 \ep^{1/\gamma}$, we obtain a different asymptotic speed $\bar s$, namely
	\[	\bar s^2=\frac{1}{C_0^\gamma (v_+-1)}.\]
	In that case, the pressure $p_-:=\lim_{\ep \to 0} p_\ep(v_-)$ is equal to $C_0^{-\gamma}$. These relations should be compared with \eqref{eq:s-limit}, \eqref{eq:fp}.
\end{rem}

\medskip

In order to fix the shift, let us consider the solution of \eqref{eq:ED-TW} with $\fv_\ep(0)=(1+v_+)/2 \in ]v_-, v_+[$ for $\ep$ small enough. Then according to the previous paragraph, we have
\[
\fv_\ep(\xi)\in \ ]v_-, v_+[ \ \subset \ ]1, v_+[ \quad \forall \xi\in \R,\ \forall \ep>0.
\]
Thus $\fv_\ep$ is uniformly bounded in $L^\infty(\R)$, and $0\leq p_\ep(\fv_\ep) \leq 1$. Looking back at \eqref{eq:ED-TW}, we deduce that $\fv_\ep$ is uniformly bounded in $W^{1,\infty}(\R)$. Therefore, using Ascoli's theorem, we infer that there exists $\bar \fv\in W^{1,\infty}(\R)$ such that up to a subsequence
\[
\ba
\fv_\ep \rightharpoonup \bar \fv \text{ in } w^*-W^{1,\infty}(\R),\\
\fv_\ep \to \bar \fv \text{ in }\mathcal C(-R,R)\quad \forall R>0. 
\ea
\]
Furthermore, $\bar \fv$ is nondecreasing,  $\bar \fv \in [1, v_+]$, and $\bar \fv(0)=(1+v_+)/2>1$. We define
\[
\bar \xi:=\inf\{ \xi \in \R,\ \bar \fv (\xi)>1\}\in [-\infty, 0[.
\]
Since $\bar \fv (\xi)>1$ for $\xi>\bar \xi$, using the above convergence result, we deduce that $p_\ep(\fv_\ep)\to 0$ in $L^\infty_{loc} (]\bar \xi, + \infty[)$. Hence we can pass to the limit in \eqref{eq:ED-TW}, and we obtain that on $]\bar \xi, + \infty[$, $\bar \fv$ is a solution of the logistic equation
\be\label{eq:barv}
\bar \fv'= \frac{\bar s}{\mu }\bar \fv (v_+-\bar \fv).
\ee
Consequently, we have an explicit formula for $\bar \fv$, namely
\[\ba
\bar \fv (\xi)& =1\quad &\forall \ \xi\leq \bar \xi,\\
\bar \fv(\xi) & =\frac{v_+}{1+a e^{-r\xi}}\quad & \forall \ \xi >\bar \xi,
\ea\]
where $r:=\bar s v_+/\mu$ and $a$ is determined by the initial condition. Since $\bar \fv(0)=(1+v_+)/2$, we have $a=(v_+-1)/(v_++1)$. This allows us to find an explicit expression for $\bar \xi$, namely
\[
\bar \xi = - \frac{\ln (v_+ +1)}{r}.
\]
Translating the profile $\bar\fv$ by $\bar \xi$ (i.e. $\xi \mapsto \xi-\bar{\xi}$) and keeping the same notation $\bar\fv$ for the new shifted profile, we recover the expression given in Proposition \ref{prop:profile}, that is
\[\ba
\bar \fv (\xi)& = 1\quad &\forall \ \xi\leq 0,\\
\bar \fv(\xi) & =\frac{v_+}{1 + (v_+ - 1) e^{-r\xi}}\quad & \forall \ \xi >0.
\ea\]

\subsection{Control in the congested zone thanks to barrier functions.}
In this paragraph, we fix the shift in $\fv_\ep$ by choosing $\fv_\ep(0)$ such that $\fv_\ep(0)-1\propto \ep^{\frac{1}{\gamma+1}}$, which will be compatible with our ansatz in the next subsection\footnote{Actually, the results of this subsection remain true as long as $v_--1<\fv_\ep(0)-1\ll \ep^{\frac{\gamma-1}{\gamma^2 }} $.}.\\
In the domain $\xi\leq 0$, we have, since $\fv_\ep$ is a monotonous function, $v_-\leq \fv_\ep(\xi)\leq \fv_\ep(0)$. 
Define 
\[
\tilde v_\ep(\zeta):= \ep^{-1/\gamma} \left(\fv_\ep (\ep^{1/\gamma} \zeta)-1\right).
\]
Then 
\begin{equation}\label{eq:tvinfty}
\tilde v_\ep (-\infty)=1,\quad \tilde v_\ep(0)=(\fv_\ep(0)-1)\ep^{-\frac{1}{\gamma}}
\end{equation}
and $\tilde v_\ep$ satisfies the ODE
\[
\tilde v_\ep'= \frac{1 + \ep^{1/\gamma} \tilde v_\ep}{\mu s_\ep}\left(s_\ep^2\ep^{1/\gamma} (1-\tilde v_\ep) + 1 - \frac{1}{\tilde v_\ep^\gamma}\right).  
\]
Now, for $\zeta\in \R_-$, we have 
\[\ba
1 + \ep^{1/\gamma}\leq 1 + \ep^{1/\gamma} \tilde v_\ep(\zeta) &\leq \fv_\ep(0),\\
1-\tilde v_\ep (\zeta)& \leq 0.
\ea
\]
Furthermore, since the function $v\mapsto v^\gamma$ is convex ($\gamma \geq 1$), for all $v>1$, we have
\[
1- \frac{1}{v^\gamma}= \frac{v^\gamma -1}{v-1} \; \frac{v-1}{v^\gamma}\geq \gamma \frac{v-1}{v^\gamma}.
\]
Therefore, for all $\zeta\leq 0$,
\[
\left| \ep^{1/\gamma} (1-\tilde v_\ep(\zeta))\right| \leq \frac{\ep^{-\frac{\gamma-1}{\gamma}} (\fv_\ep(0)-1)^\gamma}{\gamma}\left(   1 - \frac{1}{\tilde v_\ep^\gamma(\zeta)}\right).
\]
Note that thanks to the assumption on $\mathfrak v_\ep(0)$, $\ep^{-\frac{\gamma-1}{\gamma}} (\fv_\ep(0)-1)^\gamma\ll 1$.
Gathering all the inequalities, we infer that for $\zeta<0$,
\[
\bar \rho_\ep\left(   1 - \frac{1}{\tilde v_\ep^\gamma(\zeta)}\right)	\leq \tilde v_\ep'(\zeta) \leq \underline \rho_\ep \left(   1 - \frac{1}{\tilde v_\ep^\gamma(\zeta)}\right),
\]
where
\[
\bar \rho_\ep:=\frac{1 + \ep^{1/\gamma}}{\mu s_\ep} \left(1-s_\ep^2 \frac{\ep^{-\frac{\gamma-1}{\gamma}} (\fv_\ep(0)-1)^\gamma}{\gamma}\right),\quad \underline \rho_\ep :=\frac{ \fv_\ep (0)}{\mu s_\ep} ,
\]
so that $\lim_{\ep\to 0} \bar \rho_\ep= \lim_{\ep\to 0} \underline \rho_\ep= (\mu \bar s)^{-1}$.\\
Now, consider the barrier functions $\bar v_\ep$, $\underline v_\ep$, defined as solutions of the ODEs
\[
\ba
\bar v_\ep'= \bar \rho_\ep \left(   1 - \frac{1}{ \bar v_\ep^\gamma}\right),\quad
\underline v_\ep'=\underline \rho_\ep \left(   1 - \frac{1}{	\underline v_\ep^\gamma}\right) ,\\
\bar v_\ep(0)=\underline v_\ep(0)=2.
\ea
\]
According to the Cauchy-Lipschitz theorem, these two ODEs have unique solutions on $\R$ such that $\bar v_\ep>1$, $\underline{v}_\ep>1$. Furthermore, $\bar v_\ep$, $\underline v_\ep$ are increasing on $\R$ and it is easily proved that the two functions have the following asymptotic behavior
\[\ba
\lim_{\zeta \to - \infty} \bar v_\ep(\zeta)= 	\lim_{\zeta \to - \infty} \underline v_\ep(\zeta)=1,\\
\bar v_\ep(\zeta)\sim \bar \rho_\ep \zeta, \quad \underline v_\ep(\zeta)\sim \underline \rho_\ep \zeta  \quad \text{as }\zeta\to + \infty.
\ea	\]
As a consequence, there exist $\bar \zeta_\ep$, $\underline \zeta_\ep$ such that
\[
\bar v_\ep(\bar \zeta_\ep)= \underline v_\ep(\underline \zeta_\ep)=\tilde v_\ep(0)=(\fv_\ep(0)-1)\ep^{-\frac{1}{\gamma}}.
\]
Note that $\bar \zeta_\ep \sim \underline \zeta_\ep \sim  \mu \bar s (\fv_\ep(0)-1)\ep^{-\frac{1}{\gamma}}\gg 1$ since $\fv_\ep(0)-1\gg \ep^{1/\gamma} $. We also stress that as $\ep\to 0$, $\bar v_\ep$ and $\underline v_\ep$ both converge uniformly on sets of the form $]-\infty, a]$ for all $a\in \R$ towards the solution of 
\[
v'= \frac{1}{\mu \bar s} \left( 1 - \frac{1}{v^\gamma}\right),\quad
v(0)=2.
\]

We conclude our analysis of the barrier functions by investigating more precisely their behavior as $\zeta \to - \infty$. Using once again the inequalities
\[
\gamma 2^{-\gamma}\leq \gamma v^{-\gamma}<\left( 1 - \frac{1}{v^\gamma}\right) \frac{1}{v-1}\leq \gamma v^{-1}\leq \gamma \quad \forall v\in ]1, 2],
\]
we infer that there exist constants $\bar C, \underline C,\bar \sigma, \underline \sigma$, independent of $\ep$ such that
\[\ba
\bar v_\ep(\zeta)-1\leq \bar C \exp(\bar \sigma \zeta),\quad 
\underline v_\ep(\zeta) -1 \geq \underline  C \exp(\underline \sigma \zeta)\quad \forall \zeta<0.
\ea
\]
Note furthermore that it is possible to take $\underline \sigma= \underline \rho_\ep \gamma$ because of the inequality
\[
\frac{v^\gamma-1}{v^\gamma(v-1)}\leq \gamma.
\]
Indeed, we have
\[
\ba
\underline v_\ep' \leq \underline \rho_\ep \gamma (\underline v_\ep -1),\quad
\underline v_\ep(0)=2,
\ea
\]
and therefore, for all $\zeta <0$,
\[
\left(\underline v_\ep(\zeta) -1\right) \exp(-\underline \rho_\ep \gamma \zeta)\geq 1.
\]
However, concerning $\bar v_\ep$, the control on $\bar \sigma$ is not as good, because the reverse inequality reads
\[
\frac{v^\gamma-1}{v^\gamma(v-1)}\geq\gamma v^{-\gamma}\geq \gamma 2^{-\gamma} \quad \forall v\in (1,2).
\]
Of course, as $\bar v_\ep$ converges to 1, the constant in the exponential bound improves.

Let us now go back to the bounds on $\tilde v_\ep$. We have constructed Lipschitz functions $\bar F_\ep$, $\underline F_\ep$, such that $\tilde v_\ep$, $\bar v_\ep(\cdot +\bar \zeta_\ep)$, $\underline v_\ep (\cdot +\underline \zeta_\ep)$ satisfy, for all $\zeta<0$
\[
\ba 
\bar F_\ep (\tilde v_\ep)\leq 	\tilde v_\ep'\leq \underline F_\ep (\tilde v_\ep),\\
\bar v_\ep'(\cdot +\bar \zeta_\ep)= \bar F_\ep ( \bar v_\ep(\cdot +\bar \zeta_\ep)),\quad
\underline v_\ep'(\cdot +\underline \zeta_\ep)= \underline F_\ep ( \underline v_\ep(\cdot +\bar \zeta_\ep)),\\
\tilde v_\ep (0)= \bar v_\ep (\bar \zeta_\ep)=  \underline v_\ep (\underline \zeta_\ep).
\ea
\]
Classical arguments then ensure that for all $\zeta<0$,
\[
\underline v_\ep (\zeta + \underline \zeta_\ep)	\leq \tilde v_\ep(\zeta) \leq \bar v_\ep (\zeta + \bar \zeta_\ep ).
\]
Going back to the original variables, the statement of Proposition \ref{prop:profile} follows, taking $\xi_\ep:= -\ep^{\frac{1}{\gamma}}\max (\bar \zeta_\ep, \underline \zeta_\ep)$. We recall that $\ep^{\frac{1}{\gamma}} \bar \zeta_\ep \sim \ep^{\frac{1}{\gamma}} \underline \zeta_\ep\sim \mu \bar s (\mathfrak{v_\ep} (0) -1)\ll 1$, and therefore $\lim_{\ep \to 0} \xi_\ep~=~0$.

\medskip
\begin{rem}\label{rem:ext-tv}
	The above construction can easily be generalized to the case where $\lim_{\ep \to 0} p_\ep(v_-)\in ]0, +\infty[$.
	In that case, \eqref{eq:tvinfty} must be replaced by
	\[
	\tilde v_\ep (-\infty)=\frac{v_--1}{\ep^{1/\gamma}}=:\tilde v_-,\quad \tilde v_\ep(0)=(\fv_\ep(0)-1)\ep^{-\frac{1}{\gamma}}.
	\]
	After straightforward computations, one can check that
	\[
	\bar \rho_\ep:=\frac{1 + \ep^{1/\gamma}}{\mu s_\ep} \left(1-s_\ep^2 \frac{\ep^{-\frac{\gamma-1}{\gamma }} (\fv_\ep(0)-1)^\gamma}{\gamma\tilde v_-}\right) .
	\]	
	Note that the property $\lim_{\ep\to 0} \bar \rho_\ep=(\mu \bar s)^{-1}$ remains true, so that the rest of the analysis is unchanged.	
\end{rem}

\medskip
\subsection{Finer description in the transition zone.} 
We now compute a more precise asymptotic expansion of $\fv_\ep$ in the vicinity of $0$.
Indeed, there is no explicit formula for $\fv_\ep$ and therefore our purpose is to exhibit an approximation of $\fv_\ep$ which highlights its small scale dependencies in the vicinity of the transition zone $\xi = 0$.
Our goal is two-fold: firstly, since $\fv_\ep$ has $\mathcal C^1$ (and even $\mathcal C^\infty$) regularity for all $\ep>0$, it is natural to look for a $\mathcal C^1$ approximation of $\fv_\ep$, while the derivative of $\bar \fv $ has a jump at $\xi=0$. Secondly, the convergence in Paragraph \ref{ssec:qualitative-cv} is only qualitative, whereas we wish to derive a quantitative error estimate.

\medskip
We define an approximate solution $\fvapp$ by taking the following ansatz
\be\label{Ansatz-fv}
\fvapp:= \bar \fv (\xi) +\left\{
\begin{array}{cc}
	\ep^{1/\gamma} \tilde v \left(\frac{\xi-\xi^* }{\ep^{1/\gamma}}\right)~ &\text{if} ~\xi \leq 0
	, \\
	K \ep^{\frac{1}{\gamma +1}}\chi (\xi)~&\text{if}~ \xi>0 , 
\end{array}
\right.
\ee
where $\xi^*$, $K$ are real numbers that remain to be determined, together with the corrector $\tilde{v}$, $\chi\in \mathcal C^\infty_0(\R)$ is an arbitrary cut-off function such that $\chi(0)=1$ and $\chi'(0)=-1$,
and $\bar \fv$ is the profile defined in Proposition \ref{prop:profile}.
We make the following requirements on these three unknowns ($K$, $\xi^*$ and $\tilde v$):
\begin{enumerate}
	\item $\fvapp$ must be a $\mathcal C^1$ function on $\R$;
	\item $\fvapp$ must be an approximate solution of \eqref{eq:ED-TW} (in the sense that it satisfies the equation with a small, quantifiable remainder).
\end{enumerate}
We first identify $K, \xi^*$ and $\tilde{v}$, and then prove a quantitative error estimate between $\fv_\ep$ and $\fvapp$.

\medskip
\begin{rem}
	\begin{itemize}
		\item The cut-off profile in the non-congested zone $\xi \geq 0$ is merely a technical corrector, which has no actual physical or mathematical relevance.
		
		\item One important choice in the ansatz above is that $\fvapp (0 ) -1 \propto \ep^{\frac{1}{\gamma +1}}$. This choice is justified by mainly two arguments. Firstly, this ensures that $p_\ep'(v)$ remains bounded for all $v\geq \fvapp (0)$, which will be crucial in the energy estimates. Secondly, another natural ansatz would be to choose $\fvapp(\xi)= \bar \fv (\xi + \varphi_\ep)$ in the region $\xi>0$, with  $0<\varphi_\ep\ll 1$. Keeping $\varphi_\ep$ as an unkown and writing the continuity of $\fvapp$, $\fvapp'$ at $\xi=0$ leads to $\varphi_\ep \propto \ep^{\frac{1}{\gamma +1}}$ in the case $\bar s > 1$ ({\it i.e.} $v_+ < 2$), which is compatible with the ansatz \eqref{Ansatz-fv}. However, this alternative Ansatz fails when $\bar s <1$ ({\it i.e.} $v_+ > 2$), and therefore we have chosen to work only with \eqref{Ansatz-fv}.
	\end{itemize}
\end{rem}

\medskip 
\subsubsection*{Definition of the approximate solution}

Let us first identify the corrector $\tilde v$. 
Plugging the Ansatz \eqref{Ansatz-fv} for $\xi<0$ into equation \eqref{eq:ED-TW} and identifying the main order terms leads to
\[
\tilde v'= \frac{1}{\mu \bar s} \left( 1 - \frac{1}{\tilde v^\gamma}\right).
\]
We endow this ODE with an initial condition in $]1, + \infty[$, say $\tilde v(0)=2$ (this arbitrary choice will simply modify the definition of $\xi^*$ hereafter). 
Following the same reasoning as in the previous paragraph, it is easily proved that the ODE has a unique global solution $\tilde v$,  which is  increasing on $\R$. 
Furthermore, there exists a constant $\sigma>0$ such that $\tilde v$ exhibits the following asymptotic behavior at $\pm \infty$
\be\label{asympt-tildev}
\tilde v(\zeta)=1 + O\left(\exp\left( \sigma \zeta\right)\right)\quad \text{as } \zeta\to - \infty,\quad 
\tilde v(\zeta)\sim \frac{\zeta}{\mu \bar s} \quad \text{as } \zeta \to + \infty.
\ee

Now, the parameters $\xi^*$ and $K$ are determined by requiring that $\fvapp$ is continuous at $\xi=0$, with a continuous first derivative. This leads to the system
\be\label{eq:cont}
\ba
1 + \ep^{1/\gamma}\tilde{v}\left(\dfrac{ - \xi^* }{\ep^{1/\gamma}}\right)=1 + K \ep^{\frac{1}{\gamma +1}} ,\\
\tilde{v}'\left(\dfrac{- \xi^* }{\ep^{1/\gamma}}\right)= \frac{1}{\mu \bar s}- K \ep^{\frac{1}{\gamma +1}}.
\ea 
\ee
Let us set 
\[
\omega_\ep:=\tilde{v}\left(\dfrac{ - \xi^* }{\ep^{1/\gamma}}\right).
\]
Then, using the ODEs satisfied by $\tilde v$ and $\bar \fv$, the system becomes
\[
\ba
\ep^{1/\gamma} \omega_\ep=K \ep^{\frac{1}{\gamma +1}} ,\\
\frac{1}{\omega_\ep^\gamma}=\mu \bar s K \ep^{\frac{1}{\gamma +1}}. 
\ea
\]
We therefore obtain
\be\label{eq:om}
\omega_\ep:= \frac{1}{(\mu \bar s)^{\frac{1}{\gamma +1}}} \ep^{-\frac{1}{\gamma (\gamma +1)}},\quad K:= \frac{1}{(\mu \bar s)^{\frac{1}{\gamma +1}}}.
\ee
Eventually, let us compute the asymptotic behavior of $\xi^*$. Note that $\omega_\ep\gg 1$ as $\ep \to 0$, so that $\lim_{\ep\to 0 } -\ep^{-1/\gamma} \xi^*=+\infty$. As a consequence, using \eqref{asympt-tildev}, we infer that
\[
\dfrac{- \xi^* }{\mu \bar s \ep^{1/\gamma}}\sim \omega_\ep, 
\]
and thus
\be\label{asympt-xi*}
-\xi^*   
\sim (\mu \bar s)^{\frac{\gamma}{\gamma+1}} \ep^{\frac{1}{\gamma+1}}  .
\ee

\subsubsection*{Error estimate in the non-congested and transition zones}
In the vicinity of $\xi=0$, the idea is the following: we write equation \eqref{eq:ED-TW} in the form
\[
\fv_\ep'= A_\ep (\fv_\ep),
\]
where $A_\ep (v)=(\mu s_\ep)^{-1} v (s_\ep^2(v_+-v) + p_\ep(v_+)-p_\ep(v))$, and we write $\fvapp$ as an approximate solution of \eqref{eq:ED-TW}, namely
\[
\fvapp'=A_\ep (\fvapp) + r_\ep,
\]
for some small remainder $r_\ep$. We then use the form of $A_\ep$ to estimate $\fv_\ep -\fvapp$ close to $\xi=0$ through a Gronwall-type Lemma.

Let us first compute $r_\ep$. By definition of $\bar \fv$ and $\tilde v$, we have
\[
\fvapp'=\begin{cases}
\frac{1}{\mu \bar s} (1-p_\ep (\fvapp))&\text{ if } \xi < 0,\\
\frac{\bar s}{\mu} \bar \fv (v_+-\bar \fv) + K \ep^{\frac{1}{\gamma+1}} \chi'&\text{ if } \xi \geq 0,
\end{cases}
\]
so that
\[
r_\ep=\left( \dfrac{1}{\bar s} - \dfrac{\fvapp}{s_\ep}\right) \dfrac{1-p_\ep(\fvapp)}{\mu} +(v_--\fvapp) \dfrac{\fvapp s_\ep}{\mu} \quad \text{if }\xi<0,
\]
and
\[
\ba
r_\ep=&\ - \dfrac{\fvapp}{\mu s_\ep}(p_\ep (v_+) - p_\ep (\fvapp)) - \dfrac{s_\ep-\bar s}{\mu}\bar \fv (v_+ - \bar \fv)  \\
& - \dfrac{s_\ep \chi K \ep^{\frac{1}{\gamma+1}}}{\mu}\left[v_+ - 2 \bar \fv -\chi K \ep^{\frac{1}{\gamma+1}} \right]+ K \ep^{\frac{1}{\gamma+1}} \chi'\quad \text{if }\xi>0.
\ea
\]

Now, note that $\fvapp$ is bounded in $L^\infty$, uniformly in $\ep$, and that there exists a constant $C>0$ such that
\[
\ba
|\bar s -s_\ep| \leq C \ep^{1/\gamma},\\
\forall \xi \geq 0,\quad p_\ep (\fvapp(\xi)) \leq  C \ep^{\frac{1}{\gamma +1}},\\
\forall \xi \leq 0,\quad 0\leq \fvapp(\xi) -1 \leq C\ep^{\frac{1}{\gamma +1}}.
\ea
\]
Gathering these estimates, we deduce that $\|r_\ep\|_{\infty}\leq C \ep^{\frac{1}{\gamma +1}}$.

We now perform the error estimate. Without loss of generality, we can always fix the shift in $\fv_\ep$ by requiring that $(\fv_\ep- \fvapp)(0)=0$.
We treat separately the non-congested and the transition zone. Indeed, $A_\ep$ is uniformly Lipschitz in the non-congested zone, whereas the estimates on $A_\ep'(\fvapp)$ degenerate in $(\xi^*,0)$.

\medskip

$\bullet$ Non-congested zone ($\xi\geq 0$): First, recall that $\chi$ is compactly supported. As a consequence, if $\ep$ is small enough, $\fvapp$ is strictly increasing in $(0, + \infty)$, and we recall that $\fv_\ep$ is also a monotone increasing function. Hence, in the non-congested zone, we have $\fv_\ep\geq \fvapp(0)$, $\fvapp\geq \fvapp(0)$. 
Using the computations of the previous paragraph, we infer that $|p_\ep'(v)|\leq C$ for all $v\geq  \fvapp(0)$, and thus $|A_\ep'(v)|\leq C$ for all $v\geq \fvapp(0)$. We deduce that
\begin{eqnarray*}
	\left|(\fv_\ep - \fvapp)'\right| &=&\left| A_\ep (\fv_\ep) - A_\ep (\fvapp) - r_\ep\right|\\
	&\leq & C \ep^{\frac{1}{\gamma +1}} +
	C | \fv_\ep - \fvapp| .
\end{eqnarray*}
The Gronwall lemma then implies that
\[
|\fv_\ep-\fvapp| \leq C \ep^{\frac{1}{\gamma +1}} \left[ \exp(C \xi) -1\right],
\]
which leads to a  good estimate on compact intervals.

\medskip
$\bullet$ Transition zone ($\xi \in (\xi^*, 0)$): In this zone, the situation is more complicated because the derivative of the pressure might become singular. We use a bootstrap argument together with a Gronwall-type lemma to control the error $|\fv_\ep - \fvapp|$.

First, note that as long as $\xi-\xi^*\geq M \ep^{1/\gamma} $, where $M$ is some large but fixed constant, independent of $\ep$ (say $M=100$) and $\xi^*$ is defined by \eqref{eq:cont} and satisfies \eqref{asympt-xi*}, then $\fvapp(\xi)-1\sim (\mu \bar s)^{-1} (\xi-\xi^*)$. Therefore, we introduce the following bootstrap assumptions
\be\label{hyp:bootstrap}\ba
|\fv_\ep-\fvapp| \leq (4\mu \bar s)^{-1} (\xi-\xi^*),\\
\xi-\xi^*\geq M \ep^{1/\gamma}.\ea\ee
As long as the assumptions \eqref{hyp:bootstrap} are satisfied, we have
\[ 
\fvapp -1 \geq (2\mu \bar s)^{-1} (\xi-\xi^*),\quad 
\fv_\ep -1 \geq (4\mu \bar s)^{-1} (\xi-\xi^*)
\]
and therefore there exists a constant $C$, depending only on $\mu$ and $\gamma$, such that 
\[
|p_\ep'(v)|\leq \frac{C\ep}{(\xi-\xi^*)^{\gamma+1}  }\quad \forall v\in [\fv_\ep(\xi), \fvapp(\xi)].
\]
We infer that as long as the assumptions \eqref{hyp:bootstrap} are satisfied, we have
\[
|(\fv_\ep-\fvapp)'(\xi)| \leq C\left( 1 + \frac{\ep}{(\xi-\xi^*)^{\gamma+1}  }\right) | \fv_\ep-\fvapp|(\xi) + C\ep^{\frac{1}{\gamma +1}}.
\]
Note furthermore that the assumptions \eqref{hyp:bootstrap} are satisfied at $\xi=0$, and therefore, by continuity, they are also satisfied on a small interval in the vicinity of $0$. Hence, as long as the assumptions \eqref{hyp:bootstrap} are satisfied, the Gronwall Lemma ensures that
\[
\fv_\ep(\xi) - \fvapp(\xi) \leq C \ep^{\frac{1}{\gamma +1}}\int_\xi^{0} \exp \left( C (\xi'-\xi) + \frac{C\ep}{ (\xi - \xi^*)^\gamma} -  \frac{C\ep}{ (\xi' - \xi^*)^\gamma}\right) d\xi'.
\]
A similar bound holds from below. We infer that as long as the inequalities \eqref{hyp:bootstrap} hold,
\[
|\fv_\ep(\xi) - \fvapp(\xi) | \leq C \ep^{\frac{1}{\gamma +1}} |\xi| \exp \left(\frac{C\ep}{ (\xi - \xi^*)^\gamma}\right)\leq C\exp\left(\frac{C}{M^\gamma}\right) \ep^{\frac{1}{\gamma +1}} |\xi| .
\]
Without loss of generality, we choose the constant $M$ so that $\exp\left(\frac{C}{M^\gamma}\right) \leq 2$, and we obtain
\be\label{est:transition}
|\fv_\ep(\xi) - \fvapp(\xi) |\leq C \ep^{\frac{1}{\gamma +1}} |\xi|
\ee
on the interval on which assumptions \eqref{hyp:bootstrap} are valid.
Using classical bootstrap arguments, we deduce that inequalities \eqref{hyp:bootstrap}, and therefore \eqref{est:transition}, are valid as long as $\xi$ satisfies
\[
C \ep^{\frac{1}{\gamma +1}} (- \xi)\leq \frac{1}{2\mu \bar s} (\xi-\xi^*) \quad \text{and} \quad \xi-\xi^*\geq M \ep^{1/\gamma},
\]
or equivalently
\be
\xi - \xi^* \geq \max \left( M \ep^{1/\gamma}, \frac{ -C \xi^*\ep^{\frac{1}{\gamma +1}}}{ C \ep^{\frac{1}{\gamma +1}} + (2\mu \bar s)^{-1}}\right).
\ee
Using the estimate\eqref{asympt-xi*} on $\xi^*$ of the previous paragraph and the inequality $\gamma\geq 1$, we infer that the estimate \eqref{est:transition} is valid on an interval $[\xi_{min}, 0]$, where $\xi_{min}:= \xi^* + M \ep^{1/\gamma}$.

\medskip
Note that in the interval $[\xi_{min},0]$, $\fvapp$ is a good approximation, in the sense that $\fv_\ep - \fvapp$ is  smaller than all terms appearing in $\fvapp$ (namely the main order term 1 and the corrector term of order $\ep^\frac{1}{\gamma +1}$).

For $\xi\leq \xi_{min}$, the singularity in $p_\ep'$ becomes too strong to apply the Gronwall lemma. However,  we can use the control by barrier functions from the previous paragraph to estimate $\fv-\fvapp$.

\medskip
%%%%%%%%%%%%%%%%%%%%%%%%%%%%%
\section{Global well-posedness of small solutions $(W,V)$ of~\eqref{eq:NL-VW}}
\label{sec:estimates}

\medskip

The goal of this section is to prove Theorem \ref{thm:WV}, that is the existence of a global strong solution $(W,V)$ to the system 
\[
\ba
\p_t W + p_\ep (\ve + \p_x V) - p_\ep(\ve)=0,\\
\p_t V - \p_x W - \mu \p_x \ln \frac{\ve + \p_x V}{\ve}=0, \\
(W,V)_{|t=0} = (W_0,V_0),
\ea
\]
where $\ve(t,x)=\fv_\ep (x-s_\ep t)$, under a smallness assumption on $(W_0,V_0)$.
As explained in the introduction, we follow the overall strategy of \cite{vasseur2016}, tracking the dependency of all estimates with respect to $\ep$. 
Of course, the main difficulty lies in the singularity of the pressure term in the congested zones.
The main ideas are the following:

\medskip
\begin{itemize}
	\item Since we are working close to a congested profile, it is natural to investigate the stability properties of the linearized system close to this congested profile.
	Therefore we rewrite the previous system as
	\be\label{eq:NL-VW-lin}
	\ba
	\p_t W + p_\ep'(\ve)\p_x V= F_\ep(\p_x V),\\
	\p_t V - \p_x W - \mu \p_x \left(\frac{\p_x V}{\ve}\right)= G_\ep(\p_x V),
	\ea
	\ee
	where
	\be \label{eq:FG}
	\ba
	F_\ep(f):=- \left[p_\ep (\ve + f) - p_\ep (\ve ) - p_\ep'(\ve ) f\right],\\
	G_\ep(f):= \mu \p_x \left[ \ln \left(1 + \frac{f}{\ve}\right) - \frac{f}{\ve} \right].
	\ea
	\ee
	Hence the main order part of the energy and of the dissipation term is the one associated with the linearized system. 
	The nonlinear part of the operator, contained in $F_\ep$ and $G_\ep$, is then treated as a perturbation, assuming that the distance between the congested profile and the actual solution remains small enough (in a way that needs to be quantified in terms of $\ep$).
	
	\medskip
	\item In order to close the estimates thanks to a fixed point argument, we need to work in a high regularity space. 
	Therefore we differentiate the equation and derive estimates on the first-order derivatives. However, the system is not stable by differentiation, and we will need to compute some commutators. 
\end{itemize}

\medskip

\subsection{Properties of the linearized system.}
As announced above, the starting point lies in the derivation of energy estimates for the linearized system. Therefore we define the linearized operator
\[
\cLep \begin{pmatrix}
W\\ V \end{pmatrix}
= \begin{pmatrix}
p_\ep'(\ve)\p_x V\\ - \p_x W - \mu \p_x \left(\frac{\p_x V}{\ve}\right) \end{pmatrix}.
\]
The cornerstone of our analysis is the following energy estimate.

\medskip
\begin{lemma}[Energy estimates for the linearized system]
	Let $T> 0$, 
	\[
	f\in L^\infty(0,T;L^2(\R)), \quad g\in L^\infty(0,T;L^2(\R))\cap L^2(0,T;H^1(\R)),
	\] 
	and $S_f, S_g\in L^1(0,T;L^2(\R))$ such that
	\begin{equation}\label{eq:generic-lin}
	\p_t \begin{pmatrix}
	f\\ g \end{pmatrix} + \cLep \begin{pmatrix}
	f\\ g \end{pmatrix} =  \begin{pmatrix}
	S_f\\ S_g \end{pmatrix}.
	\end{equation}
	Then
	\begin{align}\label{eq:energy-lin}
	\int_\R{\left[-\dfrac{1}{p'_\ep(\ve)}{|f(T)|^2} + {|g(T)|^2} \right]}
	+ s_\ep\int_0^T\int_\R{\dfrac{p''_\ep(\ve)}{(p'_\ep(\ve))^2} \partial_x \ve |f|^2}
	+ 2\mu \int_0^T\int_\R{\dfrac{(\partial_x g)^2}{\ve}} & \nonumber\\
	= \int_\R{\left[\dfrac{-1}{p'_\ep(\ve)}{|f_0|^2} + {|g_0|^2} \right]} 
	+ 2\int_0^T\int_\R \left[S_f\frac{-f}{p_\ep(\ve)} + S_g g\right].&
	\end{align}
	\label{lem:est-linearized}
\end{lemma}

\begin{proof}
	To get~\eqref{eq:energy-lin}, we test Equation~\eqref{eq:generic-lin} against $\begin{pmatrix} \dfrac{-f}{p'_\ep(\ve)} \\ g \end{pmatrix}$ and we obtain
	\begin{align*}
	\frac{1}{2}	\dfrac{d}{dt}\int_\R{\left[-\dfrac{|f|^2}{p'_\ep(\ve)} + {|g|^2}\right]} 
	- \int_\R{\partial_t\left(-\dfrac{1}{p'_\ep(\ve)}\right)\dfrac{|f|^2}{2} } 
	- \int_\R{\partial_xg f} - \int_\R{\partial_x f g}& \\ 
	- \mu\int_\R{\partial_x \left(\dfrac{\partial_x g}{\ve}\right) g}
	= \int_\R \left[S_f\frac{-f}{p_\ep(\ve)} + S_g g\right]. &
	\end{align*}
	Using then integration by parts and $\partial_t \ve = -s_\ep \partial_x \ve$, this equality is rewritten as
	\begin{align*}
	& \dfrac{d}{dt}\int_\R{\left[-\dfrac{|f|^2}{p'_\ep(\ve)} + {|g|^2}\right]} 
	+ s_\ep\int_\R{\dfrac{p''_\ep(\ve)}{(p'_\ep(\ve))^2}\partial_x \ve{|f|^2} } 
	+ 2\mu \int_\R{\dfrac{(\partial_x g)^2}{\ve}}\\
	& = 2\int_\R \left[S_f\frac{-f}{p_\ep(\ve)} + S_g g\right].
	\end{align*}
	which leads to~\eqref{eq:energy-lin} after integration in time.
\end{proof}

\medskip

We will apply  Lemma \ref{lem:est-linearized} with $(f,g)=\partial_x^k (W,V)$ and with $k=0,1,2$. 
Therefore it is important to compute the commutator of $\cLep$ with the differential operator $\p_x$.

\medskip
\begin{lemma}[Properties of the commutator {$[\cLep, \p_x]$}]
	For all $(f,g)\in L^2_{\text{loc}} (\R_+, H^1(\R))^2  $,
	\[
	[\cLep, \p_x]\begin{pmatrix}
	f\\g
	\end{pmatrix}= \begin{pmatrix}
	-\p_x \ve p_\ep''(\ve) \p_x g\\
	-\mu \p_x \left(\frac{\p_x \ve}{\ve^2}\p_x g\right)\end{pmatrix}
	\]
	and
	\[[\cLep, \p_x^2 ]\begin{pmatrix}
	f\\g
	\end{pmatrix}= 2 [\cLep, \p_x]\begin{pmatrix}
	\partial_xf\\\partial_xg
	\end{pmatrix}- \begin{pmatrix}
	\p_x (\p_x \ve p_\ep''(\ve))\p_x g\\ - \mu \p_x \left( \p_x \left(\frac{\p_x \ve}{\ve^2 }\right) \p_x g\right) 
	\end{pmatrix}.\]
	As a consequence, we have the following bound:
	there exists a constant $C_1$ depending only on $\mu, v_+$ and $\gamma$ such that for all $\delta>0$, for all $T>0$, 
	\begin{equation}\label{eq:estim-commutator}
	\left|\int_0^T\int_{\R}{[\cLep, \p_x]\begin{pmatrix}f\\g\end{pmatrix} \cdot \begin{pmatrix}\dfrac{-\partial_xf}{p'_\ep(\ve)}\\ \partial_xg\end{pmatrix}} \right| \\\leq\delta \int_0^T\int_{\R}\partial_x \ve |\partial_x f|^2
	+ \frac{C_1}{\delta}\ep^{-2/\gamma} \int_0^T\int_{\R}{|\partial_x g|^2} .
	\end{equation}
	\label{lem:commutator}
\end{lemma}

Lemma \ref{lem:commutator} will be proved in Paragraph \ref{ssec:commutator}.

\medskip
\begin{rem}{\label{rem:commutator}}
	Let us stress  that the term $ \p_x \ve p_\ep''(\ve) \p_x g$ in the commutator $[\cLep, \p_x]$ is responsible for a loss of $\ep^{2/\gamma}$ in the second integral of~\eqref{eq:estim-commutator}.
	It means that we will have to multiply our energy estimate at each iteration by $\ep^{2/\gamma}$. In other words, our total energy will be
	\[
	\sum_{k=0}^2 \ep^{2k/\gamma} \int_\R{\left[\dfrac{-1}{p'_\ep(\ve)}{|\p_x^k W (t)|^2}+ {|\p_x^k V(t)|^2} \right]}.
	\]
\end{rem}

\subsection{Construction of global strong solutions of \eqref{eq:NL-VW}.}
In this paragraph, we construct global smooth solutions of \eqref{eq:NL-VW} under a smallness assumption. 
Following Lemma \ref{lem:est-linearized}, we derive successive estimates on $(V,W)$ and their space derivatives up to order 2. 
Hence, we define
\[\ba 
E_k(t;V,W):= \int_\R{\left[\dfrac{-1}{p'_\ep(\ve)}{|\p_x^k W (t)|^2}+ {|\p_x^k V(t)|^2} \right] dx},\\
D_k(t;V,W):= \int_\R \partial_x\ve |\p_x^k W|^2  dx
+  \int_\R{{(\partial_x^{k+1} V)^2} dx}.
\ea
\]
Note that 
\[
\frac{p_\ep''(\ve)}{p_\ep'(\ve)^2}= \frac{(\gamma+1) (\ve-1)^\gamma}{\gamma\ep }= \frac{\gamma+1}{\gamma p_\ep(\ve)}\geq \frac{\gamma+1}{\gamma}.
\]
As a consequence, there exists a constant $C_2$, depending only on $\gamma$, $\mu$ and $v_+$, such that for $\ep$ small enough, for all $(W,V)\in H^k(\R) \times H^{k+1}(\R) $,
\[
D_k(t; V, W) \leq C_2 \left(s_\ep\int_\R{\dfrac{p''_\ep(\ve)}{(p'_\ep(\ve))^2} \partial_x\ve {|\p_x^k W|^2}}
+ 2\mu \int_\R{\dfrac{(\partial_x^{k+1} V)^2}{\ve}}\right).
\]
The goal is to prove, by a fixed point argument, existence and uniqueness of global smooth solutions of \eqref{eq:NL-VW}, under the assumption that $E_k(0)$ is small enough for $k=0,1,2$. 
Given the couple $(W_1,V_1)$, we introduce the following system
\be\label{eq:syst-fixed-point}
\ba
&\p_t \begin{pmatrix} W_2\\ V_2 \end{pmatrix} + \cLep\begin{pmatrix} W_2\\ V_2 \end{pmatrix}= \begin{pmatrix} F_\ep(\p_x V_1)\\ G_\ep(\p_x V_1) \end{pmatrix}\\
&(W_2, V_2)_{|t=0}=(W_0, V_0)
\ea
\ee
and the application
\[
\mathcal A^\ep :( W_1, V_1) \in \cX \mapsto ( W_2, V_2) \in \cX,
\]
where
\[
\cX:= \{ (W,V) \in L^\infty(\R_+;  H^2(\R))^2 ;\ D_k(t;W,V)\in L^1(\R_+) \text{ for }k=0,1,2 \}.
\]
We endow $\cX$ with the norm
\begin{equation}\label{eq:df-norm}
\|(W,V)\|_{\cX}^2:= \sup_{t\in [0,+\infty[} \left[\sum_{k=0}^2 c^k \ep^{2k/\gamma} \left[E_k(t,W(t),V(t)) + \int_0^t D_k(s,W(s),V(s))\:ds\right]
\right],
\end{equation}
where $c$ is a constant to be determined, which is meant to be small but independent of $\ep$,
and for $\delta>0$, we denote by $B_\delta$ the ball 
\begin{equation}
B_\delta = \{(W,V)\in \cX,\ \|(W,V)\|_{\cX}< \delta \ep^{\frac{5}{2\gamma}}\}.
\end{equation}

The result of Theorem \ref{thm:WV} will be achieved with the proof of the following proposition.

\medskip
\begin{prop}
	Assume that
	\[
	E_0(0; W_0, V_0) + \ep^{2/\gamma} E_1(0; W_0, V_0) + \ep^{4/\gamma} E_2(0; W_0, V_0) \leq  \delta_0^2 \ep^{\frac{5}{\gamma}}.
	\]
	for some $\delta_0 > 0$.
	There exist two positive constants $\delta^*$ and $c_0$, depending only on $v_+, \mu$ and $\gamma$, such that if $0 < \delta_0 < \delta^*$, $0 < c < c_0$, then there exists $\delta=\delta(\delta_0, v_+, \mu, \gamma)$ such that
	\begin{itemize}
		\item The ball $B_{\delta}$ is stable by $\mathcal A^\ep$.
		
		\item The application $\mathcal A^\ep$ is a contraction on $B_{\delta}$.
	\end{itemize}
	As a consequence, $\mathcal A^\ep$ has a unique fixed point in $B_{\delta}$.
	\label{prop:fixed-point}
\end{prop}

\medskip
Note that we are able to prove a global result. This comes from the fact that our system is dissipative, which allows us to circumvent the use of the Gronwall Lemma.

As a preliminary, let us recall that
\begin{equation}
\left\| \frac{1}{\ve-1}\right\|_\infty \leq \ep^{-1/\gamma},\label{est:v-infty}
\end{equation}
so that $|p_\ep'(\ve)|\leq \gamma \ep^{-1/\gamma}$. Additionally, differentiating \eqref{eq:ED-TW}, we have
\begin{equation}\label{eq:d2v_ep}
|\partial^2_{x}\ve| \leq C \ep^{-1/\gamma} \partial_x \ve.
\end{equation}
Hence, if $(W,V) \in \cX$ then for $k=0,1$, $m=0,1,2$,
\begin{align}
\|\partial_x^m V\|_{L^\infty_t(L^2_x)}&\leq  C \ep^{-m/\gamma}  \| (W,V)\|_{\cX},\quad
\|\partial_x^m W\|_{L^\infty_t(L^2_x)}\leq  C \ep^{-\frac{1+2m}{2\gamma}}  \| (W,V)\|_{\cX},\\
\|\partial_x^k V\|_{L^\infty_{t,x}} & \leq C\|\partial_x^k V\|_{L^\infty_t L^2_x}^{1/2}\|\partial^{k+1}_{x} V\|_{L^\infty_t L^2_x}^{1/2} \nonumber\\
&\leq C \ep^{-\frac{k}{2\gamma}}\| (W,V)\|_{\cX}^{1/2} \ep^{-\frac{k+1}{2\gamma}}\| (W,V)\|_{\cX}^{1/2}\nonumber \\&\leq C \ep^{-\frac{2k+1}{2\gamma}}  \| (W,V)\|_{\cX},\label{est:dxV-infty}\\ \|\p_x^{m+1} V\|_{L^2_{t,x}}&\leq C \ep^{-\frac{m}{\gamma}} \| (W,V)\|_{\cX},\\
\|\partial_x^k W\|_{L^\infty_{t,x}} & \leq  C \ep^{-\frac{k+1}{\gamma}}  \| (W,V)\|_{\cX},\label{est:dxW-infty},
\end{align}
where the constant $C$ depends only on $c$ and $\gamma$. We will use this remark repeatedly when estimating the source term $(F_\ep, G_\ep)$.\\
Note in addition that if
\[
\|(W,V)\|_{\cX} \leq \delta \ep^{\frac{5}{2\gamma}}
\]
then the following inequality holds
\begin{equation}\label{eq:bound-dxV}
\|\partial_x V\|_{L^\infty_{t,x}} \leq C \delta \ep^{\frac{5}{2\gamma} - \frac{3}{2\gamma}} \leq C\delta \ep^{\frac{1}{\gamma}} < \ep^{\frac{1}{\gamma}}
\end{equation}
provided that $\delta$ is small enough.
In other words, for sufficiently small $\delta$ the perturbation $v = \ve + \partial_x V$ will never reach the critical value $v^*=1$.

In order to prove Proposition \ref{prop:fixed-point}, we rely on the energy estimate from Lemma \ref{lem:est-linearized}, and we treat the right-hand side $(F_\ep(\p_x V_1), G_\ep(\p_x V_1))$, defined in \eqref{eq:FG}, as a perturbation that we estimate thanks to Lemmas \ref{lem:F-ep_G-ep} and \ref{lem:diff-F-ep_G-ep} below. 
The largest part of the proof is devoted to the stability of the ball $B_{\delta}$ by the application $\mathcal A_\ep$. We derive successive estimates for $E_k(t; W_2,V_2)$ in terms of $\|(W_1,V_1)\|_{\cX}$ and $ \|(W_2,V_2)\|_{\cX}$. 
Note that we cannot close the estimates before performing the estimate on $E_2$. Furthermore, when addressing the bound on $E_1$ (resp. $E_2$), we will use the commutator result of Lemma \ref{lem:commutator} together with the control on $\p_x V$ (resp. $\p_x^2 V$) in $L^2_{t,x}$ coming from lower order estimates.
Eventually, we prove that $\mathcal A_\ep$ is a contraction on $B_{\delta}$.

\subsubsection*{Tools and heuristics for the control of non-linear terms}
One of the main technical difficulties of the estimates comes from the nonlinear terms $F_\ep$ and $G_\ep$. 
We will rely on the following Lemma (see also Lemma \ref{lem:diff-F-ep_G-ep}):

\medskip
\begin{lemma}
	Let us write $G_\ep(f)= \mu \p_x (H_\ep(f))$, where 
	\[
	H_\ep= \ln \left(1 + \frac{f}{\ve}\right) - \frac{f}{\ve} ,
	\]
	and recall that $F_\ep(f)=-\left[ p_\ep(v_\ep +f) - p_\ep (v_\ep) - p_\ep'(v_\ep) f\right]$.
	Provided that $|f|\leq \frac{\ep^{1/\gamma}}{2}$,  there exists a constant $C$, independent of $\ep$, such that the following estimates hold:
	\begin{align*}
	| F_\ep(f)| &\leq C p_\ep(\ve)\frac{f^2}{(\ve-1)^2}, \\%\label{eq:bound-F}\\
	|\partial_x F_\ep(f)| & \leq C \partial_x \ve |p'_\ep(\ve)|\frac{f^2}{(\ve-1)^2}  +  C p_\ep(\ve)\frac{|f|\; |\p_x f| }{(\ve-1)^2} , \\
	|\partial^2_{x} F_\ep(f)| 
	& \leq C \ep^{-1/\gamma} \partial_x \ve |p'_\ep(\ve)|  \frac{f^2}{(\ve-1)^2} 
	+ C p_\ep(\ve)\frac{(\p_x f)^2}{(\ve -1)^2}
	+ C p_\ep(\ve) \frac{|f|\; |\p_x^2 f|}{(\ve -1)^2}, \nonumber
	\end{align*}
	and
	\begin{align*}
	|H_\ep(f)| &\leq C |f|^2, \\ %\label{eq:bound-H}\\
	|\partial_x H_\ep(f)| & \leq C |f| |\partial_{x}f| + C|f|^2, \\
	|\partial^2_{x} H_\ep(f)| & \leq C |f||\partial^{2}_{x}f| + C \Big(|f| + |\partial_{x}f| \Big)|\partial_{x}f|
	+ C\Big(1 + |\partial^2_{x}\ve| \Big)|f|^2.
	\end{align*}
	\label{lem:F-ep_G-ep}
\end{lemma}

When we perform $L^2$ estimates, taking into account Lemma \ref{lem:F-ep_G-ep}, we need to control terms of the type
\[
\int_0^t \int_{\R}f_\ep[\ve] | \p_x^k U_1|\; |  \p_x^l U_1| \; | \p_x^m U_2|,
\]
with $k,l,m \in \{0,1,2\}$, $U_i = V_i $ or $W_i$, and $f_\ep[\ve] $ is a function of $\ve$ and its derivatives. In order to guide the reader, we establish the following (ordered) rules to control such terms:
\begin{enumerate}
	\item If a term contains a factor of the form $\p_x \ve  \p_x^k W_i$, this factor is controlled through $D_k$;
	\item The (remaining) term with the smallest number of derivatives is controlled in $L^\infty_x$, and the other(s) in $L^2_x$;
	\item Note that $\|\p_x^k V\|_{L^2_x}$ for $k\geq 1$ could be controlled either through $E_k$ or through $D_{k-1}$. 
	Nevertheless we will always use $D_{k-1}$ to ensure uniform-in-time estimates.
\end{enumerate}

\medskip 
\subsubsection*{Estimate for $k=0$}  
For $k=0$, the estimate from Lemma~\ref{lem:est-linearized} applied to $f=W_2$, $g=V_2$, $S_f = F_\ep(\partial_xV_1)$ and $S_g = \mu \partial_xH_\ep(\partial_x V_1)$, entails that for all $t\geq 0$,
\begin{align*}
& E_0(t; W_2, V_2) + C_2^{-1} \int_0^t D_0(s; W_2, V_2) \:ds \\
& \leq E_0(0; W_0, V_0) + 2\int_0^t \left(\left|\int F_\ep(\p_x V_1) \frac{1}{p_\ep'(\ve)} W_2\right| + \left| \mu \int \p_x H_\ep(\p_x V_1)  V_2\right|\right).
\end{align*}
Using estimate \eqref{est:v-infty} and Lemma \ref{lem:F-ep_G-ep}, we infer that 
\begin{align*}
&\int_0^\infty \int_\R \left(\left| F_\ep(\p_x V_1) \frac{1}{p_\ep'(\ve)} W_2\right| + \left| \mu  \partial_xH_\ep(\p_x V_1)  V_2\right|\right)\\
&\leq C \int_0^\infty \int_\R \left( \frac{(\p_x V_1)^2}{\ve -1} |W_2| +   |\p_x V_1||\partial^2_{x}V_1| |V_2| +   |\p_x V_1|^2 |V_2|\right)\\
&\leq C \left( \ep^{-1/\gamma} \|W_2\|_{L^\infty_{t,x}} \|\p_x V_1\|_{L^2_{t,x}}^2 + \|V_2\|_{L^\infty_{t,x}} \|\p_x V_1\|_{L^2_{t,x}}\|\p^2_{x} V_1\|_{L^2_{t,x}}  + \|V_2\|_{L^\infty_{t,x}} \|\p_x V_1\|_{L^2_{t,x}}^2\right).
\end{align*}
Using estimates \eqref{est:dxV-infty} and \eqref{est:dxW-infty} together with the assumption $(W_1, V_1)\in B_{\delta}$, we infer that the right-hand side  above is bounded by
\[
C\ep^{-\frac{2}{\gamma}} \| (W_2, V_2)\|_\cX\| (W_1, V_1)\|_\cX^2 
\leq C \delta^2 \ep^{\frac{3}{\gamma}} \| (W_2, V_2)\|_\cX.
\]
Therefore we obtain
\be\label{eq:k=0}
\sup_{t\in [0,+\infty[}\left(E_0(t; W_2, V_2) + \int_0^t D_0(s; W_2, V_2) \:ds \right) 
\leq C \left( E_0(0; W_0, V_0)  + \delta^2 \ep^{\frac{3}{\gamma}}\| (W_2, V_2)\|_\cX\right).
\ee

%%%%%%%%%%%%%%%%%%%
\subsubsection*{Estimate for $k=1$} 
We apply now Lemma~\ref{lem:est-linearized} to 
\begin{equation*}
f=\partial_xW_2, 
\quad g=\partial_xV_2, 
\end{equation*}

\begin{align*}
\begin{pmatrix}S_f \\S_g \end{pmatrix}
= \begin{pmatrix}\partial_xF_\ep(\partial_xV_1) 
\\ \mu \partial^2_{x}H_\ep(\partial_x V_1) \end{pmatrix}
+ [\cLep,\partial_x] \begin{pmatrix} W_2 \\ V_2 \end{pmatrix},
\end{align*} 
and get, for all $t\geq 0$,
\begin{align}\label{eq:k=1-0}
& E_1(t; W_2, V_2) +  C_2^{-1}\int_0^t D_1(s; W_2, V_2) \:ds \nonumber\\
& \leq E_1(0; W_0, V_0) 
+ 2 \int_0^t \left|\int_\R\dfrac{1}{p_\ep'(\ve)} \partial_x F_\ep(\p_x V_1)  \partial_xW_2\right| 
+ 2 \mu \int_0^t \left| \int_\R \p_{x} H_\ep(\p_x V_1)  \partial^2_{x} V_2\right|\\
& \quad	+ 2\int_0^t \int_{\R} \left| {[\cLep, \p_x]\begin{pmatrix}W_2\\ V_2\end{pmatrix} \cdot \begin{pmatrix}\dfrac{-\partial_xW_2}{p'_\ep(\ve)} \\ \partial_x V_2\end{pmatrix}} \right| \nonumber
\end{align}
The term involving the commutator is controlled via inequality~\eqref{eq:estim-commutator}, and is bounded by
\[
\frac{ C_2^{-1}}{2}\int_0^t D_1(s; W_2, V_2) \:ds + 8C_1C_2\ep^{-2/\gamma}\int_0^tD_0(s,W_2,V_2)\:ds. 
\]
The first integral can be absorbed in the left-hand side of \eqref{eq:k=1-0}.

By using Lemma~\ref{lem:F-ep_G-ep} we can estimate the integrals of nonlinear terms of the right-hand side of \eqref{eq:k=1-0}, namely
\begin{align*}
& \int_0^t \left|\int_\R\dfrac{1}{p_\ep'(\ve)} \partial_x F_\ep(\p_x V_1)  \partial_xW_2\right| 
+ \mu \int_0^t \left| \int_\R \p_{x} H_\ep(\p_x V_1)  \partial^2_{x} V_2\right|\\
& \leq 
C\int_0^t \int_\R{\partial_x\ve \left|\dfrac{\partial_x V_1}{\ve-1}\right|^2 |\partial_xW_2|} 
+ C\int_0^t \int_\R{  \frac{|\p_x V_1|\; | \p^2_{x} V_1|}{\ve -1}  |\partial_xW_2|} \\
& \quad + C\int_0^t \int_\R{|\partial_x V_1||\partial^2_{x}V_1||\partial^2_{x}V_2|} 
+ C\int_0^t \int_\R{|\partial_x V_1|^2|\partial^2_{x}V_2|}.
\end{align*}
We follow the guidelines stated at the beginning of the proof, and use estimates \eqref{est:v-infty}-\eqref{est:dxV-infty} repeatedly. 
We infer that these nonlinear terms are bounded by
\begin{align*}
& \int_0^t \left|\int_\R\dfrac{1}{p_\ep'(\ve)} \partial_x F_\ep(\p_x V_1)  \partial_xW_2\right| 
+ \mu \int_0^t \left| \int_\R \p_{x} H_\ep(\p_x V_1)  \partial^2_{x} V_2\right|\\
& \leq  C\bigg[\left(\int_0^t D_1(s; W_2, V_2)ds\right)^{1/2} \ep^{-2/\gamma} \|\p_x V_1\|_{L^\infty_{t,x}}\|\p_x V_1\|_{L^2_{t,x}} \\
& \quad + \ep^{-1/\gamma}\|\p_x W_2\|_{L^\infty_{t,x}} \|\p_x V_1\|_{L^2_{t,x}} \|\p_x^2 V_1\|_{L^2_{t,x}} \\
& \quad + \|\p_x V_1\|_{L^\infty_{t,x}}\|\p_x^2 V_1\|_{L^2_{t,x}} \|\p_x^2 V_2\|_{L^2_{t,x}} +  \|\p_x V_1\|_{L^\infty_{t,x}}\|\p_x V_1\|_{L^2_{t,x}} \|\p_x^2 V_2\|_{L^2_{t,x}} \bigg]\\
& \leq C\ep^{- \frac{9}{2\gamma}} \|(W_2, V_2)\|_\cX \|(W_1,V_1)\|_{\cX}^2 \\
& \leq C \delta^2 \ep^{\frac{1}{2\gamma}}  \|(W_2, V_2)\|_\cX .
\end{align*}
Therefore
\begin{align*}
&\sup_{t\geq 0} \left[E_1(t; W_2, V_2) + \frac{C_2^{-1}}{2}\int_0^t D_1(s; W_2, V_2)\:ds\right]\\
&\leq  E_1(0; W_0, V_0)  + 8 C_1C_2\ep^{-2/\gamma} \int_0^\infty D_0(s; W_2, V_2) + C\delta^2 \ep^{\frac{1}{2\gamma}} \|(W_2, V_2)\|_\cX.
\end{align*}
Note that without loss of generality, we can always choose $C_2\geq 1/2$, so that the above inequality becomes
\begin{align*}
&\sup_{t\geq 0} \left[E_1(t; W_2, V_2) + \int_0^t D_1(s; W_2, V_2)\:ds\right]\\
& \leq 2 C_2 E_1(0; W_0, V_0)  + 16C_1C_2^2\ep^{-2/\gamma} \int_0^\infty D_0(s; W_2, V_2) + C\delta^2 \ep^{\frac{1}{2\gamma}}   \|(W_2, V_2)\|_\cX.
\end{align*}

Hence, choosing $c\leq c_0\leq (32C_1C_2^2)^{-1}$ and using~\eqref{eq:k=0}
\begin{align}\label{eq:k=1}
& \sup_{t\in \R_+}\left[E_0(t; W_2, V_2) + c\ep^{2/\gamma}E_1(t; W_2, V_2)\right]
+ \int_{\R_+}\left(D_0(s; W_2, V_2) +c\ep^{2/\gamma}D_1(s; W_2, V_2) \right)\:ds \nonumber\\
& \leq C \left[ E_0(0; W_0, V_0)  + \ep^{2/\gamma}E_1(0; W_0, V_0)   + \delta^2 \ep^{\frac{5}{2\gamma}}  \|(W_2, V_2)\|_\cX\right].
\end{align}

\medskip
%%%%%%%%%%%%%%%%%%%%%%%%
\subsubsection*{Estimate for $k=2$}
We apply once again Lemma~\ref{lem:est-linearized} to 
\[
f=\partial^2_{x}W_2, 
\quad g=\partial^2_{x}V_2
\]
with the source term (see Lemma~\ref{lem:commutator})
\begin{align*}
\begin{pmatrix}S_f \\S_g \end{pmatrix}
& = \begin{pmatrix}\partial^2_{x}F_\ep(\partial_xV_1) \\ \mu \partial^3_{x}H_\ep(\partial_x V_1) \end{pmatrix}
+ [\cLep,\partial^2_{x}] \begin{pmatrix}W_2 \\V_2 \end{pmatrix} \\
& = \begin{pmatrix}\partial^2_{x}F_\ep(\partial_xV_1) \\ \mu \partial^3_{x}H_\ep(\partial_x V_1) \end{pmatrix}
+ 2[\cLep,\partial_{x}] \begin{pmatrix}\partial_xW_2 \\\partial_xV_2 \end{pmatrix}
- \begin{pmatrix} \partial_x( p_\ep''(\ve) \p_x \ve) \p_x V_2\\ -\mu \p_x \left(\partial_x\left(\frac{\p_x \ve}{\ve^2}\right)\p_x V_2\right)\end{pmatrix}.
\end{align*} 
Observe first that from~\eqref{eq:estim-commutator}, we have
\begin{align*}
& 4 \left |\int_0^t\int_{\R}{[\cLep,\partial_{x}] \begin{pmatrix}\partial_xW_2 \\\partial_xV_2 \end{pmatrix} 
	\cdot \begin{pmatrix}\dfrac{-\partial^2_{x}W_2}{p'_\ep(\ve)} \\ \partial^2_{x}V_2\end{pmatrix}} \right| \nonumber \\
& \leq \dfrac{C_2^{-1}}{4}\int_0^t\int_{\R}\partial_x \ve |\partial^2_{x}W_2|^2
+ 84 C_1 C_2 \ep^{-2/\gamma}\int_0^t\int_{\R}{|\partial^2_{x}V_2|^2} .
\end{align*}
Concerning the additional commutator term, we have on the one hand, using the control~\eqref{eq:d2v_ep} on $\partial^2_{x}\ve$,
\begin{align*}
&2 \left| \int_0^t \int_{\R }   \partial_x( p_\ep''(\ve) \p_x \ve) \p_x V_2 \dfrac{-\partial^2_{x}W_2}{p'_\ep(\ve)}  \right| 	\\
& \leq 2 \int_0^t \int_{\R } (\gamma+1)\left( (\gamma+2)  \frac{(\p_x \ve)^2}{(\ve -1)^2} + \frac{|\partial^2_{x} \ve|}{\ve -1}\right) | \partial_xV_2|\; | \partial^2_{x}W_2|\\
& \leq  C_3 \ep^{-2/\gamma} \int_0^t \int_{\R } \p_x \ve  | \partial_xV_2|\; | \partial^2_{x}W_2|\\
& \leq  \dfrac{C_2^{-1}}{4}\int_0^t\int_{\R}\partial_x \ve |\partial^2_{x}W_2|^2 + C_3^2 C_2 \ep^{-4/\gamma} \int_0^tD_0(s; W_2,V_2)ds,
\end{align*}
for some constant $C_3$ depending only on $\mu, v_+$ and $\gamma$. On the other hand
\begin{align*}
& 2 \left| \int_0^t \int_{\R } \mu \p_x \left(\partial_x\left(\frac{\p_x \ve}{\ve^2}\right)\p_x V_2\right)\p^2_{x} V_2\right|\\
& \leq 2 \mu   \int_0^t \int_{\R } \left|\partial_x\left(\frac{\p_x \ve}{\ve^2}\right)\right| \; |\p_x V_2| \; | \partial^3_{x}V_2|\\
& \leq  \dfrac{C_2^{-1}}{4}\int_0^t\int_{\R}| \partial^3_{x}V_2|^2 + C \ep^{-2/\gamma}  \int_0^tD_0(s; W_2,V_2)ds.
\end{align*}
We now address the nonlinear terms.
From Lemma~\ref{lem:F-ep_G-ep}, we have, concerning the remainder involving $F_\ep$,
\begin{align*}
&\int_0^t\left|\int_\R\dfrac{1}{p_\ep'(\ve)} \partial^2_{x} F_\ep(\p_x V_1)  \partial^2_{x}W_2\right|  \\
&\leq C \ep^{-1/\gamma} \int_0^t \int_{\R} \p_{x} \ve \frac{(\p_x V_1)^2}{(\ve-1)^2} |\p^2_{x} W_2|\\
& \quad + C\int_0^t \int_{\R} \frac{(\p^2_{x} V_1)^2}{\ve-1} |\p^2_{x} W_2|\\
& \quad + C \int_0^t \int_{\R} \frac{1}{\ve-1}  | \p_{x} V_1|\; | \p_{x}^3 V_1|\;  |\p^2_{x} W_2|.	
\end{align*}
Using the inequalities \eqref{est:v-infty}-\eqref{est:dxV-infty} together with classical Sobolev embeddings, we infer that the remainder involving $F_\ep$ is bounded by
\begin{align*}
&C\ep^{-3/\gamma}\left(	\int_0^t D_2(s; W_2, V_2) ds\right)^{1/2}\| \p_x V_1\|_{L^\infty_{t,x}} \| \p_x V_1\|_{L^2_{t,x}} \\
& \quad + C \ep^{-1/\gamma} \int_0^t\left( \|\p_x^2 V_1\|_{L^4_x} ^2 + \|\p_x V_1\|_{L^\infty_x} \|\p_x^3 V_1\|_{L^2_x}\right) \|\p_x^2 W_2\|_{L^2_x}\\
&\leq  C \ep^{-3/\gamma} \times \ep^{-2/\gamma} \| (W_2, V_2)\|_{\cX} \times  \ep^{-\frac{3}{2\gamma} } \| (W_1, V_1)\|_{\cX}^2\\
& \quad  + C  \ep^{-1/\gamma} \| \p_x^2W_2\|_{L^\infty_t(L^2_x)} \left( \|\p_x^2 V_1\|_{L^2_{t,x}}^{3/2}\| \p_x^3 V_1\|_{L^2_{t,x}}^{1/2} + \|\p_x V_1\|_{L^2_{t,x}}^{1/2} \| \p_x^2 V_1\|_{L^2_{t,x}}^{1/2} \|\p_x^3 V_1\|_{L^2_{t,x}}\right)\\
&\leq  C \ep^{-\frac{13}{2\gamma}}  \| (W_2, V_2)\|_{\cX} \| (W_1, V_1)\|_{\cX}^2 \\
& \leq C \delta^2 \ep^{-\frac{3}{2\gamma}}   \| (W_2, V_2)\|_{\cX}. 
\end{align*}

Now we deal with the integral coming from $H_\ep$, namely
\begin{align*}
& \mu \int_0^t \int_\R {|\partial^2_{x}H_\ep(\partial_x V_1)| |\partial^3_{x}V_2|} \\
& \leq  C \int_0^t \int_\R{|\partial_x V_1| | \p_x^3 V_1|\;  |\partial^3_{x}V_2|} 
+ C \int_0^t \int_\R{\big(|\partial_x V_1|+|\partial^2_{x}V_1|\big)|\partial^2_{x}V_1|  |\partial^3_{x}V_2|} \\
& \quad + C \int_0^t \int_\R{\big(1 + |\partial^2_{x}\ve|\big)|\partial_x V_1|^2|\partial^3_{x}V_2|} \\
& \leq  C \|\p_x^3 V_2\|_{L^2_{t,x}} \Big(\|\p_x V_1\|_{L^\infty_{t,x}} (\|\p_x^3 V_1\|_{L^2_{t,x}} +\|\p_x^2 V_1\|_{L^2_{t,x}}  + \ep^{-1/\gamma} \|\p_x V_1\|_{L^2_{t,x}} ) \\
& \hspace{2cm} + \| \p_x^2 V_1\|_{L^2_{t,x}}^{3/2} \| \p_x^3 V_1\|_{L^2_{t,x}}^{1/2}   \Big)\\
& \leq C \ep^{-\frac{11}{2\gamma}} \|(W_1, V_1)\|_{\cX}^2 \|(W_2, V_2)\|_{\cX} \\
& \leq C \delta^2 \ep^{-\frac{1}{2\gamma}}  \|(W_2, V_2)\|_{\cX}.
\end{align*}
Gathering all the terms, we obtain, for all $t\geq 0$,
\begin{align*}
& E_2(t;W_2,V_2) + C_2^{-1}\int_0^t{D_2(s;W_2,V_2)\ ds}\\
& \leq E_2(0;W_0,V_0)+ \frac{C_2^{-1}}{2}\int_0^t{D_2(s;W_2,V_2)\ ds} + 84 C_1 C_2 \ep^{-2/\gamma} \int_0^t D_1(s; W_2, V_2)\:ds\\
& \quad + (C_3^2 C_2 \ep^{-4/\gamma} + C \ep^{-2\gamma}) \int_0^t D_0(s; W_2, V_2)\:ds
+ C \delta^2 \ep^{-\frac{3}{2\gamma}} \|(W_2, V_2)\|_{\cX}.
\end{align*}
Therefore, for $\ep$ small enough and recalling that $C_2\geq 1$,
\begin{align*}
&\sup_{t\geq 0}\left[E_2(t;W_2,V_2) + \int_0^t{D_2(s;W_2,V_2)\ ds}\right]\\
& \leq  2 C_2E_2(0;W_0,V_0)+  C \delta^2 \ep^{-\frac{3}{2\gamma}}\|(W_2, V_2)\|_{\cX}\\
&\quad + 128 C_1 C_2^2 \ep^{-2/\gamma} \int_0^t D_1(s; W_2, V_2)\:ds+ 4 C_3^2 C_2^2 \ep^{-4/\gamma} \int_0^t D_0(s; W_2, V_2)\:ds.
\end{align*}
Now, choose $c_0=\frac{1}{2}\min( (128C_1C_2^2)^{-1}, (4 C_3^2 C_2^2 )^{-1/2})$. If $c\leq c_0$, using \eqref{eq:k=1}, we obtain
\begin{align}\label{eq:k=2}
&\sup_{t\geq 0} \sum_{k=0}^2 c^k \ep^{2k/\gamma} \left[E_k(t;W_2,V_2) + \int_0^t{D_k(s;W_2,V_2)\ ds}\right]\nonumber\\
& \leq  C  \sum_{k=0}^2  \ep^{2k/\gamma}E_k(0;W_0,V_0) + C \delta^2 \ep^{\frac{5}{2\gamma}}\|(W_2, V_2)\|_{\cX}. 
\end{align}
Recalling the definition of the $\|\cdot\|_{\cX}$ norm and using  Young's inequality, we infer that
\[
\|(W_2, V_2)\|_{\cX}^2\leq C_4  \sum_{k=0}^2  \ep^{2k/\gamma}E_k(0;W_0,V_0) + C\delta^4 \ep^{\frac{5}{\gamma}},
\]
where the constant $C_4$ depends only on $\gamma, v_+$ and $\mu$. 
Hence, if initially 
\[
E_0(0; W_0 ,V_0) + \ep^{2/\gamma}E_1(0; W_0, V_0) + \ep^{4/\gamma} E_2(0;W_0,V_0)
< \delta_0^2 \ep^{\frac{5}{\gamma}}
\]
with $\delta_0, \delta$ such that $C_4 \delta_0^2 + C\delta^4 < \delta^2 <1$,
we ensure by~\eqref{eq:k=2} that 
\begin{equation}
\|(W_2,V_2)\|_{\mathcal{X}}^2 \leq \delta^2 \ep^{\frac{5}{\gamma}}.
\end{equation}
Therefore the ball $B_{\delta}$ is stable by $\mathcal A^\ep$.

\subsubsection*{$\mathcal A_\ep$ is a contraction}

Consider $(W_1,V_1)\in B_{\delta}$, $(W_1',V_1')\in B_{\delta}$, and the associated solutions $(W_2,V_2)=\mathcal A_\ep (W_1,V_1)$, $(W_2',V_2')=\mathcal A_\ep (W_1',V_1')$. Then $(W_2-W_2', V_2-V_2')$ is a solution of \eqref{eq:generic-lin} with the source term $S_f=F_\ep(\p_x V_1) - F_\ep(\p_x V_1')$, $S_g=\mu \p_x \left[H_\ep(\p_x V_1) - H_\ep(\p_x V_1')\right]$. 
The next lemma provides bounds on the new source terms. 

\begin{lemma}
	For all $f_1, f_2\in H^2_{\text{loc}}(\R)$ such that $|f_1|+ |f_2|\leq \frac{\ep^{1/\gamma}}{2}$,
	\begin{align*}
	| F_\ep(f_1) - F_\ep(f_2)| 
	&\leq C \frac{p_\ep(\ve)}{(\ve-1)^2}|f_1-f_2| (|f_1| + |f_2|)\\
	|\partial_x (F_\ep(f_1) - F_\ep(f_2))|
	& \leq C \frac{p_\ep(\ve)}{(\ve-1)^2} \Bigg[\frac{\p_x \ve}{\ve-1}|f_1-f_2| (|f_1| + |f_2|) \\
	& \hspace{2cm} + |\p_x (f_1 - f_2)| \: |f_1| + | \p_x f_2| \; |f_1-f_2|\Bigg]\\
	|\partial_x^2  (F_\ep(f_1) -  F_\ep(f_2))|&\leq  C \frac{p_\ep(\ve)}{(\ve-1)^2} \Bigg[\frac{\ep^{-1/\gamma}\p_x \ve}{\ve-1}|f_1-f_2| (|f_1| + |f_2|)  \\
	& \hspace{2cm} + \frac{1}{\ve-1} \big(|\partial_x(f_1-f_2)||f_1| + |f_1 - f_2| |\partial_x f_1|\big) \\
	& \hspace{2cm} + |\p_x (f_1 - f_2)|\;  (|\p_x f_1| +|\p_x f_2|)  + |f_1-f_2|\; |\p_x^2 f_1| \\
	& \hspace{2cm} + |\p_x^2 (f_1-f_2)| \; |f_2| + \frac{1}{\ve-1}(\p_x f_2)^2|f_1-f_2|
	\Bigg]
	\end{align*}
	and
	\begin{align*}
	|H_\ep(f_1) - H_\ep(f_2)| &\leq  C |f_1-f_2| (|f_1| + |f_2|)\\
	|\p_x (H_\ep(f_1) - H_\ep(f_2))|&\leq  C \left[ | f_1|\; |\p_x (f_1-f_2)| + |f_1|\; |f_1-f_2| + (| \p_x f_2| + |f_2|) |f_1-f_2| \right]\\
	|\p_x^2  (H_\ep(f_1) - H_\ep(f_2))|&\leq C |f_1| ((1+ |\p_x^2 \ve|)  |f_1-f_2| + |\p_x (f_1 - f_2)| + | \p_x^2 (f_1-f_2)| )\\
	& \quad + C  ((1+ |\p_x^2 \ve|)  |f_2| + |\p_x f_2| + | \p_x^2 f_2| )|f_1-f_2|  \\
	& \quad + C  |\p_x (f_1 - f_2)|\;  (|\p_x f_1| +|\p_x f_2|)+ C (|f_2|^2 + |\p_x f_2|^2) |f_1-f_2|.
	\end{align*}
	\label{lem:diff-F-ep_G-ep}
\end{lemma}

We postpone the proof of the lemma to Paragraph \ref{ssec:F-ep_G-ep}.
Using these estimates, the control of $E_k(t; W_2-W_2', V_2-V_2')$ for $k=0,1$ follows the same lines as the one of  $E_k(t; W_2, V_2)$ above. In particular, since $(W_2-W_2', V_2-V_2')_{|t=0}=0$, we find that for $c\leq c_0\leq (32C_1 C_2^2)^{-1}$,
\begin{multline*}
\sup_{t\in \R_+} \sum_{k=0}^1c^k\ep^{2k/\gamma}   \left[E_k(t; W_2-W_2', V_2-V_2') + \int_0^t D_k(s;   W_2-W_2', V_2-V_2') \:ds\right]\\
\leq C \delta \|(W_1-W_1', V_1-V_1')\|_{\cX}   \|(W_2-W_2', V_2-V_2')\|_{\cX}.
\end{multline*}
However, concerning the estimate for $k=2$, there is a difference, stemming from the term
\begin{align*}
C\frac{p_\ep(\ve)}{(\ve-1)^3 }(\p_x^2 V_1')^2|\p_x V_1 - \p_x V_1'|\\
(\text{resp. } C((\p_x V_1')^2 +  (\p_x^2 V_1')^2 ) |\p_x V_1 - \p_x V_1'|)
\end{align*}
coming from $\p_x^2(F_\ep (\p_x V_1)-  F_\ep (\p_x V_1'))$ (resp. from $\p_x^2(H_\ep (\p_x V_1)-  H_\ep (\p_x V_1'))$), see Lemma \ref{lem:diff-F-ep_G-ep}. As a consequence, following the estimates of the case $k=2$ above, we find that for all $t\geq 0$,
\begin{align*}
&E_2(t; W_2-W_2', V_2-V_2') + C_2^{-1} \int_0^t D_2(s;   W_2-W_2', V_2-V_2')\:ds\\
&\leq \frac{C_2^{-1}}{2}\int_0^t{D_2(s;W_2-W_2',V_2-V_2')\ ds} + 64 C_1 C_2 \ep^{-2/\gamma} \int_0^t D_1(s; W_2-W_2',V_2-V_2')\:ds\\
& \quad +(C_3^2 C_2 \ep^{-4/\gamma} + C \ep^{-2\gamma}) \int_0^t D_0(s; W_2-W_2',V_2-V_2')\:ds\\
& \quad + C \delta \ep^{ - \frac{4}{\gamma}} \|(W_1-W_1', V_1-V_1')\|_{\cX}  \|(W_2, V_2)\|_{\cX}\\
& \quad + C \int_0^t\int_{\R} \frac{1}{(\ve-1)^2} (\p_x^2 V_1')^2|\p_x V_1 - \p_x V_1'|\; | \p_x^2 (W_2-W_2')|\\
& \quad + C \int_0^t\int_{\R}((\p_x V_1')^2 +  (\p_x^2 V_1')^2 ) |\p_x V_1 - \p_x V_1'|\; | \p_x^3 (V_2-V_2')| . 
\end{align*}
The first additional nonlinear term is bounded as follows, using \eqref{est:v-infty}-\eqref{est:dxV-infty}
\begin{align*}
& \int_0^t\int_{\R} \frac{1}{(\ve-1)^2} (\p_x^2 V_1')^2|\p_x V_1 - \p_x V_1'|\; | \p_x^2 (W_2-W_2')|\\
&\leq  C \ep^{-2/\gamma} \|\p_x V_1 - \p_x V_1'\|_{L^\infty_{t,x}} \| \p_x^2 (W_2-W_2')\|_{L^\infty_t(L^2_x)} \|\p_x^2  V_1'\|_{L^2_t(L^4_x)}^2\\
&\leq  C \ep^{-2/\gamma} \|\p_x V_1 - \p_x V_1'\|_{L^\infty_{t,x}} \| \p_x^2 (W_2-W_2')\|_{L^\infty_t(L^2_x)}\|\p_x^2 V_1'\|_{L^2_{t,x}}^{3/2}  \|\p_x^3 V_1'\|_{L^2_{t,x}}^{1/2}  \\
&\leq  C \ep^{-\frac{17}{2\gamma}} \|(W_1-W_1', V_1-V_1')\|_{\cX}\|(W_2-W_2', V_2-V_2')\|_{\cX}\|(W_1', V_1')\|_{\cX}^2\\
&\leq  C \delta^2 \ep^{-\frac{7}{2\gamma}}  \|(W_1-W_1', V_1-V_1')\|_{\cX}\|(W_2-W_2', V_2-V_2')\|_{\cX}.
\end{align*}
For the second additional nonlinear term, we have in a similar way
\begin{align*}
&\int_0^t\int_{\R}((\p_x V_1')^2 +  (\p_x^2 V_1')^2 ) |\p_x V_1 - \p_x V_1'|\; | \p_x^3 (V_2-V_2')|\\
&\leq  C \|\p_x V_1 - \p_x V_1'\|_{L^\infty_{t,x}} \|\p_x^3 (V_2 -  V_2')\|_{L^2_{t,x}} (\| \p_x V_1'\|_{L^4_{t,x}}^2 +   \| \p_x^2 V_1'\|_{L^4_{t,x}}^2)\\
&\leq  C \|\p_x V_1 - \p_x V_1'\|_{L^\infty_{t,x}} \|\p_x^3 (V_2 -  V_2')\|_{L^2_{t,x}} \| \p_x V_1'\|_{L^2_{t,x}}^{1/2} \| \p_x V_1'\|_{L^\infty_t (L^2_x )} \| \p_x^2 V_1'\|_{L^2_{t,x}}^{1/2} \\
& \quad +   C \|\p_x V_1 - \p_x V_1'\|_{L^\infty_{t,x}} \|\p_x^3 (V_2 -  V_2')\|_{L^2_{t,x}}\| \p_x^2 V_1'\|_{L^2_{t,x}}^{1/2} \| \p_x^2 V_1'\|_{L^\infty_t (L^2_x )} \| \p_x^3 V_1'\|_{L^2_{t,x}}^{1/2} \\
&\leq  C \ep^{-7/\gamma} \|(W_1-W_1', V_1-V_1')\|_{\cX}\|(W_2-W_2', V_2-V_2')\|_{\cX}\|(W_1', V_1')\|_{\cX}^2\\
&\leq  C \delta^2 \ep^{- \frac{2}{\gamma}} \|(W_1-W_1', V_1-V_1')\|_{\cX}\|(W_2-W_2', V_2-V_2')\|_{\cX}.
\end{align*}

As a consequence, gathering all the terms, we infer that
\[
\|(W_2-W_2', V_2-V_2')\|_{\cX} 
\leq C \delta \|(W_1-W_1', V_1-V_1')\|_{\cX},
\]
and therefore, $\mathcal A_\ep$ is a contraction on $B_{\delta}$ for $\delta < \delta^*$ small enough. 
This concludes the proof of Proposition \ref{prop:fixed-point}.

\medskip

%%%%%%%%%%%%%%%%%%%%%%%%%%%%%%%
\section{Asymptotic stability of the profiles $(\ue,\ve)$}{\label{sec:stability}}

\medskip

Our goal in this paragraph is to prove the existence and uniqueness of solutions of the original system \eqref{eq:NS-ep} (rather than the integrated system \eqref{eq:NL-VW}), and to investigate their long-time behavior. At this stage, we have proved the following:

\medskip
\begin{itemize}
	\item If $(u,v)$ is a smooth solution of \eqref{eq:NS-ep} such that $v-\ve, w-\we\in L^1_0(\R)$ for all times, then we can write system \eqref{eq:NL-VW} for the integrated quantities $(W,V)$;
	
	\medskip
	\item If the initial energy of the system is small enough, there exists a unique strong solution of \eqref{eq:NL-VW} (see Proposition \ref{prop:fixed-point}).
	
\end{itemize}
Therefore, our strategy is as follows: we start from the unique solution of \eqref{eq:NL-VW}. Under additional assumptions on the initial data, we derive bounds on $u-\ue, v-\ve$. In particular, we prove that if initially $(u-\ue)_{|t=0}, (v-\ve)_{|t=0}, (w-\we)_{|t=0}\in L^1_0$, this property remains true for all times. These local $L^1$ bounds  rely on arguments similar to the ones used by Haspot in \cite{haspot2018}. This justifies the equivalence between the original system \eqref{eq:NS-ep} and the integrated system \eqref{eq:NL-VW}. Eventually, we prove that $(u-\ue)(t)\to 0$ and $(v-\ve)(t)\to 0$ as $t\to \infty$ in $L^2\cap L^\infty(\R)$.

\medskip

\subsection*{Initial perturbations}
Let $u_0, \ v_0$ satisfy the hypotheses of Theorem \ref{thm:estimates} and introduce the integrated quantity $U_0$ such that $\partial_x U_0(\cdot) = u_0(\cdot)-\ue(0,\cdot)$.
We have then (recall that $v_{\ep|t=0}=\fv_\ep$)
\[
W_0 = U_0 - \mu \dfrac{\partial_x V_0}{\fv_\ep} - 
\mu \left[\ln\left(1+\dfrac{\partial_x V_0}{\fv_\ep}\right) - \dfrac{\partial_x V_0}{\fv_\ep}\right].
\]
By assumption $(U_0,V_0)\in H^2(\R)\times H^2(\R)$ and condition~\eqref{eq:init-1} in Theorem~\ref{thm:WV} is fulfilled, that is 
\[
\sum_{k=0}^2 \ep^{\frac{2k}{\gamma}}\int_{\R}{\left[ \dfrac{|\partial^k_x W_0|^2}{-p'_\ep(\fv_\ep)} + |\partial^k_x V_0|^2 \right]} \leq \delta_0^2 \ep^{\frac{5}{\gamma}}.
\]
Moreover, since $V_0 \in H^3(\R)$, we have
\begin{align}
\|(u-\ue)(0)\|_{L^2_x}
& = \left\|\partial_x W_0 + \mu \partial_x\left(\dfrac{\partial_x V_0}{\fv_\ep}\right)
+ \partial_x H_\ep(\partial_x V_0) \right\|_{L^2_x} \nonumber\\
& \leq \|\partial_x W_0\|_{L^2_x} + C (\|\partial^2_{x} V_0\|_{L^2_x} + \|\partial_{x} V_0\|_{L^2_x}) + \|\partial_x H_\ep(\partial_x V_0)\|_{L^2_x} \nonumber\\
& \leq C\Big[\|\partial_x W_0\|_{L^2_x} + \|\partial^2_{x} V_0\|_{L^2_x} + \|\partial_{x} V_0\|_{L^2_x} \nonumber\\
& \hspace{2cm} + \|\partial_x V_0\|_{L^\infty_x} 
\big(\|\partial_x V_0\|_{L^2_x} + \|\partial^2_{x} V_0\|_{L^2_x}\big) \Big] \nonumber \\
& \leq C \delta_0 \ep^{\frac{1}{2\gamma}},
\end{align}
and
\begin{align}
\|\partial_x(u-\ue)(0)\|_{L^2_x}
& \leq \|\partial^2_{x} W_0\|_{L^2_x} 
+ C\|\partial^2_{x} H_\ep(\partial_x V_0)\|_{L^2_x} \nonumber\\
& \quad + C (\|\partial^3_{x} V_0\|_{L^2_x} + \|\partial^2_{x} V_0\|_{L^2_x} + \ep^{-1/\gamma}\|\partial_x V_0\|_{L^2_x}) \nonumber \\
& \leq \|\partial^2_{x} W_0\|_{L^2_x} + C\big(1+\|\partial_x V_0\|_{L^\infty_x}\big)\|\partial^3_{x} V_0\|_{L^2_x} 
+ C\|\partial^2_{x} V_0\|_{L^4_x}^{2} \nonumber\\
& \quad + C\ep^{-1/\gamma}\|\partial_x V_0\|_{L^\infty_x}\|\partial_x V_0\|_{L^2_x} \nonumber \\
& \leq \|\partial^2_{x} W_0\|_{L^2_x} + C\big(1+\|\partial_x V_0\|_{L^\infty_x}\big)\|\partial^3_{x} V_0\|_{L^2_x} 
+ C\|\partial^2_{x} V_0\|_{L^2_x}^{3/2} \|\partial^3_xV_0\|_{L^2_x}^{1/2} \nonumber\\
& \quad + C\ep^{-1/\gamma}\|\partial_x V_0\|_{L^\infty_x}\|\partial_x V_0\|_{L^2_x} \nonumber \\
& \leq C \delta_0 + C \|\partial^3_{x} V_0\|_{L^2_x}
\end{align}
using the result of Lemma~\ref{lem:F-ep_G-ep}.

\medskip 

\subsection*{Stability of the velocity profile $\ue$}
The perturbation $u-\ue$ satisfies the parabolic equation
\begin{align}
\partial_t(u -\ue) - \mu\partial_x\left(\frac{1}{v}\partial_x(u-\ue)\right) 
& = - \partial_x(p_\ep(v) -p_\ep(\ve)) 
+ \mu \partial_x\left(\left(\frac{1}{v}-\frac{1}{\ve}\right)\partial_x \ue \right),
\label{eq:upertub}
\end{align} 
where $v=\ve + \partial_x V$, and $(W,V)$ is a solution of \eqref{eq:NL-VW}.

\medskip
\begin{lemma}\label{lem:stability-u}
	Assume that initially $(U_0,V_0)\in H^2(\R)\times H^3(\R)$ is such that~\eqref{eq:init-1} is satisfied by the couple $(W_0,V_0)$
	and consider the solution $(W,V) \in B_\delta \subset \cX$ of~\eqref{eq:NL-VW} given by Theorem~\ref{thm:WV}.
	Then there exists a unique regular solution $u-\ue$ to~\eqref{eq:upertub} which is such that
	\begin{align}
	u-\ue \in \mathcal{C}([0,+\infty); H^1(\R)) \cap L^2([0,+\infty),H^2(\R)), \quad 
	\partial_t(u-\ue) \in L^2([0,+\infty)\times \R).
	\end{align}	
	Moreover the following estimate holds
	\begin{equation}\label{eq:L2-estimate-u}
	\sup_{t \in \R_+}\left[ \|(u-\ue)(t)\|_{H^1}^2 + \int_0^t{\|\partial_x(u-\ue)(s)\|_{H^1}^2 \ ds} \right] \leq
	C \big(\|(u-\ue)(0)\|_{H^1}^2 + \delta^2 \ep^{\frac{1}{\gamma}}\big).
	\end{equation}
\end{lemma}

\medskip
\begin{proof}
	Under the initial condition~\eqref{eq:init-1}, Theorem~\ref{thm:WV} applies and yields the existence of a unique couple $(W,V) \in B_\delta$. 
	For this $V$, we define $v =\ve +\partial_x V$. Then $\inf v \geq 1 + c \ep^{1/\gamma}$ for some  positive constant $c$, and using \eqref{est:dxV-infty}, we also have $\|v\|_{L^\infty(\R_+, W^{1,\infty}(\R))} \leq C$.

	First, we test the equation~\eqref{eq:upertub} against $u-\ue$ to get
	\begin{align*}
	& \int_{\R}{\dfrac{|(u-\ue)(t)|^2}{2}} -  \int_{\R}{\dfrac{|(u-\ue)(0)|^2}{2}}
	+ \mu \int_0^t\int_{\R}{\dfrac{1}{v}|\partial_x(u-\ue)|^2} \\
	& =  \int_0^t\int_{\R}{(p_\ep(v) -p_\ep(\ve)) \partial_x(u-\ue)} 
	- \mu \int_0^t\int_{\R}{\left(\frac{1}{v}-\frac{1}{\ve}\right)\partial_x \ue \ \partial_x(u-\ue)} 
	\end{align*}
	where the right-hand side can be estimated as follows, 	using the relation $\partial_x \ue = -s_\ep \partial_x \ve$ to bound $|\partial_x \ue|$
	\begin{align*}
	\big|{\rm RHS}\big|
	& \leq \dfrac{\mu}{2}\int_0^t\int_{\R}{\dfrac{1}{v}|\partial_x(u-\ue)|^2} 
	+ C \int_0^t\int_{\R}{\big|p_\ep(\ve + \partial_x V) -p_\ep(\ve)\big|^2} \\
	& \quad + C \int_0^t\int_{\R}{\left|\frac{1}{\ve + \partial_x V}-\frac{1}{\ve}\right|^2 \big|\partial_x \ue\big|^2}\\
	& \leq \dfrac{\mu}{2}\int_0^t\int_{\R}{\dfrac{1}{v}|\partial_x(u-\ue)|^2}
	+ C\ep^{-2/\gamma}\|\partial_x V\|_{L^2_{t,x}}^2 + C \|\partial_x V\|_{L^2_{t,x}}^2 \\
	& \leq \dfrac{\mu}{2}\int_0^t\int_{\R}{\dfrac{1}{v}|\partial_x(u-\ue)|^2}
	+ C\delta^2\ep^{3/\gamma} .
	\end{align*}
Therefore we have
	\begin{equation}\label{eq:u-0}
	\sup_{t \in \R_+}\left[ \|(u-\ue)(t)\|_{L^2_x}^2 +\frac{\mu}{2} \int_0^t{\|\partial_x(u-\ue)(s)\|_{L^2_x}^2 \ ds} \right] 
	\leq \|(u-\ue)(0)\|_{L^2_x}^2 + C \delta^2 \ep^{3/\gamma}.
	\end{equation}
	To obtain an estimate at the next order, we test equation~\eqref{eq:upertub} against $-\partial^2_{x}(u-\ue)$:
	\begin{align*}
	& \int_{\R}{\dfrac{|\partial_x(u-\ue)(t)|^2}{2}} -  \int_{\R}{\dfrac{\partial_x|(u-\ue)(0)|^2}{2}}
	+ \mu \int_0^t\int_{\R}{\dfrac{1}{v}|\partial^2_{x}(u-\ue)|^2} \\
	& =  \mu\int_0^t\int_{\R}{\dfrac{\partial_x v}{v^2}\partial_x(u-\ue) \ \partial^2_{x}(u-\ue)}
	+ \int_0^t\int_{\R}{\partial_x(p_\ep(\ve + \partial_x V) -p_\ep(\ve)) \ \partial^2_{x}(u-\ue)} \\
	& \quad - \mu \int_0^t\int_{\R}{\partial_x\left(\left(\frac{1}{\ve + \partial_x V}-\frac{1}{\ve}\right)\partial_x \ue\right)\ \partial^2_{x}(u-\ue)}.
	\end{align*}
	As previously we estimate the right-hand side by means of Cauchy-Schwarz and Young's inequalities
	\begin{align*}
	\big|{\rm RHS}\big|
	& \leq \dfrac{\mu}{2}\int_0^t\int_{\R}{\dfrac{1}{v}|\partial^2_{x}(u-\ue)|^2} 
	+ C \int_0^t\int_{\R}{|\partial_{x}(u-\ue)|^2} \\
	& \quad	+ C \ep^{-4/\gamma}\|\partial_x V\|_{L^2_{t,x}}^2 + C \ep^{-2/\gamma}\|\partial^2_{x}V\|_{L^2_{t,x}}^2 \\
	& \quad	+ C \ep^{-2/\gamma}\|\partial_x V\|_{L^2_{t,x}}^2 + C \|\partial^2_{x}V\|_{L^2_{t,x}}^2 \\
	& \leq \dfrac{\mu}{2}\int_0^t\int_{\R}{\dfrac{1}{v}|\partial^2_{x}(u-\ue)|^2} 
	+ C \int_0^t\int_{\R}{|\partial_{x}(u-\ue)|^2}
	+ C \delta^2 \ep^{1/\gamma}
	\end{align*} 
	using the relation $\partial_x \ue = - s_\ep \partial_x \ve$ to deduce that $|\partial^2_{x}\ue| \leq C \ep^{-1/\gamma}$.  
	Combining this inequality with the previous estimate~\eqref{eq:u-0} we obtain~\eqref{eq:L2-estimate-u}.
	As a consequence, we also deduce from equation~\eqref{eq:upertub} that
	\begin{align*}
	& \|\partial_t (u-\ue)\|_{L^2_{t,x}} \\
	& \leq C\Big(\|\partial^2_{x}(u-\ue)\|_{L^2_{t,x}} + \|\partial_{x}(u-\ue)\|_{L^2_{t,x}}	
	+ \ep^{-2/\gamma}\|\partial_x V\|_{L^2_{t,x}} + \ep^{-1/\gamma}\|\partial^2_x V\|_{L^2_{t,x}}
	\Big) \\
	& \leq C \delta \ep^{\frac{1}{2\gamma}}.
	\end{align*}
	The existence and uniqueness of $u$ derives classically from these a priori estimates. 
\end{proof}

\medskip
\begin{rem}
	Combining equation \eqref{eq:upertub} with (the $x$ derivative of) \eqref{eq:NL-VW}, we infer that the quantity $w-u + \mu \p_x (\ln v)$ is a  solution in the sense of distributions of the parabolic equation
	\[
	\p_t \left(w-u + \mu \p_x (\ln v)\right) - \mu \p_x \left[ \frac{1}{v} \p_x \left(w-u + \mu \p_x (\ln v)\right)\right]=0.
	\]
	Furthermore, by definition of $W_0$, we also have $\left(w-u + \mu \p_x (\ln v)\right)_{|t=0}=0$. As a consequence, 
	\[
	w-u + \mu \p_x (\ln v)=0 \quad \text{for a.e. } t>0, \ x\in \R.
	\]
\end{rem}

\subsection*{$L^1$ estimates}
The previous lemma is based on the existence and uniqueness of a regular $v = \ve + \partial_x V$ and thus on the passage to the integrated quantities $(W,V)$.
Nevertheless, we did not justify the equivalence between the system
\begin{subnumcases}{\label{eq:vw}}
\partial_t(w-\we) + \partial_x(p_\ep(v) - p_\ep(\ve)) = 0 \\
\partial_t(v-\ve) - \partial_x(w-\we) - \mu \partial^2_{x}\ln\frac{v}{\ve} = 0
\end{subnumcases}
and the system~\eqref{eq:NL-VW} satisfied by the integrated quantities.
Initially, we assumed that $(w_0-\we(0),v_0-\ve(0)) \in L^1_0(\R)$ to justify the introduction of $(W_0,V_0)$ (note that the assumptions of Theorem \ref{thm:estimates}, namely $u_0 - \ue(0) \in W^{1,1}_0$ and $v_0-\ve(0) \in W^{2,1}_0\cap H^2$, ensure that $w_0-\we(0)\in L^1_0$). The goal of this paragraph is to prove that this property remains true for all times.
This result relies on a combination of estimates on the both velocities $u-\ue$ and $w-\we$, similar to the estimates in \cite{haspot2018}.

\medskip
\begin{lemma}\label{lem:uw-L1}
	Assume that the conditions of the previous lemma are satisfied. 
	Suppose in addition that
	\[
	u_0-\ue(0) \in L^1_0(\R), \quad v_0-\ve(0) \in W^{1,1}_0(\R).
	\]
	Then for all times $t \geq 0$, $(u-\ue)(t)$ and $(w - \we)(t)$ belong to $L^1_0(\R)$ and
	\begin{align} \label{eq:L^1-uw}
	\|(u-\ue)(t)\|_{L^1_x} + \|(w-\we)(t)\|_{L^1_x}
	\leq C_\ep \Big[\|u_0-\ue(0)\|_{L^1_x} + \|w_0-\we(0)\|_{L^1_x}\Big] \ e^{C_\ep t}
	\end{align}
	where the constant $C_\ep$ tends to $+\infty$ as $\ep \rightarrow 0$.
\end{lemma}

\medskip

\begin{proof}
	The functions $u-\ue$ and $w-\we$ satisfy the equations
	\begin{align}
	&\partial_t(u -\ue) - \mu\partial_x\left(\frac{1}{v}\partial_x(u-\ue)\right) 
	= - \partial_x(p_\ep(v) -p_\ep(\ve)) 
	+ \mu \partial_x\left(\left(\frac{1}{v}-\frac{1}{\ve}\right)\partial_x \ue \right) \label{eq:u},\\
	&\partial_t(w-\we) = - \partial_x(p_\ep(v) - p_\ep(\ve)) . \label{eq:w}
	\end{align}
	For $n> 0$, we introduce $j_n \in \mathcal{C}^2(\R)$ defined by 
	\[
	j_n(z) = \sqrt{z^2 + \frac{1}{n}} - \sqrt{\frac{1}{n}} \quad \forall z\in \R
	\]
	which is a smooth, convex approximation of the function $r \mapsto |r|$ as $n\rightarrow +\infty$.
	Note that $j'_n(z) = z\big(\sqrt{z^2+1/n}\big)^{-1}$ is an approximation of the sign function.
	Testing equations \eqref{eq:u}-\eqref{eq:w} against $j'_n(u-\ue)$ and $j'_n(w-\we)$ respectively, we infer that
	\begin{align*}
	& \int_{\R}{\partial_t\big[j_n(u-\ue) + j_n(w-\we) \big]} 
	+ \mu \int_\R{\frac{1}{v}j''_n(u-\ue)|\partial_x(u-\ue)|^2} \\
	& = - \int_\R{\partial_x(p_\ep(v) -p_\ep(\ve)) \big[j'_n(u-\ue) + j'_n(w-\we)\big]} \\
	& \quad + \mu\int_\R{\partial_x\left(\left(\frac{1}{v}-\frac{1}{\ve}\right)\partial_x \ue \right)j'_n(u-\ue)}.
	\end{align*}
	Since $j''_n > 0$, the second integral of the left-hand side has a positive sign.
	On the other hand, since the profile $(\ve,\ue)$ satisfies $\partial^k_x \ve,\ \partial^k_x \ue \in L^1(\R)$, $k\geq 1$, the right-hand side can be controlled by
	\begin{align*}
	|{\rm RHS}|
	& \leq \int_\R{\big|\partial_x(p_\ep(v) -p_\ep(\ve)\big|}
	+ \mu\int_\R{\left|\partial_x\left(\left(\frac{1}{v}-\frac{1}{\ve}\right)\partial_x \ue \right)\right|}\\
	& \leq \int_\R{\big|p'_\ep(v)\partial_x(v-\ve) + (p'_\ep(v)-p'_\ep(\ve))\partial_x \ve\big|}
	+ \mu\int_{\R}{\left|\partial_x\left(\frac{1}{v}-\frac{1}{\ve}\right)\right| |\partial_x\ue|} \\
	& \quad	+ \mu\int_{\R}{\left|\frac{1}{v}-\frac{1}{\ve}\right| |\partial^2_{x}\ue|} \\
	& \leq  C \ep^{-1/\gamma} \left( \| \p_x v \|_{L^1_x} + \| \p_x v_\ep \|_{L^1_x}\right)\\
	&\quad +
	C \|v-\ve\|_{L^\infty(\R_+, W^{1,\infty}(\R))}\Big(  \ep^{-2/\gamma}\|\partial_x \ve\|_{L^1_x} + \|\partial_{x} \ue\|_{L^1_x}
	+ \|\partial^2_{x} \ue\|_{L^1_x} \Big)
	\end{align*}  
	where
	\begin{equation*}
	\partial_x v = \frac{v}{\ve}\partial_x \ve + \frac{v}{\mu}\big[(u-\ue) -(w-\we)\big]
	\end{equation*}
	so that
	\[
	\|\partial_x v\|_{L^1_x} 
	\leq C\Big(\|\partial_x\ve\|_{L^1_x} + \|u-\ue\|_{L^1_x} + \|w-\we\|_{L^1_x}\Big).
	\]
	Hence
	\begin{align*}
	& \int_{\R}{\big[j_n(u-\ue)(t) + j_n(w-\we)(t) \big]} - \int_{\R}{\big[|u-\ue|(0) + |w-\we|(0) \big]} \\
	& \leq C_\ep\left(\|\partial_x \ve\|_{L^1} +\|\partial^2_{x} \ve\|_{L^1} +\int_0^t \int_{\R}{\big[|u-\ue| + |w-\we|\big] }\right)
	\end{align*}
	where we have used the fact that $j_n(r) \leq |r|$.
	Passing to the limit $n\rightarrow +\infty$ and using Fatou's lemma, we finally obtain~\eqref{eq:L^1-uw} thanks to a Gronwall inequality.
	Since the equations~\eqref{eq:u}-\eqref{eq:w} are conservative, we ensure that
	\[
	\int_{\R}{(u-\ue)(t)} = 0, \quad \int_{\R}{(w-\we)(t)} = 0 \quad \forall t\geq 0.
	\]
\end{proof}

\medskip
Observe that the previous lemma gives $L^1$ bounds on $u-\ue$ and $w-\we$ but not on $v-\ve$.
Since $v-\ve$ satisfies 
\[
\partial_t(v-\ve) -\partial_x(u-\ue) = 0,
\]
the derivation of a $L^1$ estimate requires a control of $\partial_x(u-\ue)$ in $L^1_x$.

\medskip
\begin{lemma}\label{lem:v-L1}
	Assume that the conditions of the previous lemmas are satisfied. 
	Suppose in addition that
	\[
	\partial_x(u_0-\ue(0)) \in L^1(\R), \quad \partial_x(w_0-\we(0)) \in L^1(\R).
	\]
	Then for all times $t \geq 0$, $(v-\ve)(t)$, $\partial_x(u-\ue)(t)$ and $\partial_x(w - \we)(t)$ belong to $L^1_0(\R)$ and
	\begin{align} \label{eq:L^1-px-uw}
	&\|(v-\ve)(t)\| _{L^1_x} + \|(u-\ue)(t)\|_{W^{1,1}_x} + \|(w-\we)(t)\|_{W^{1,1}_x} \nonumber \\
	&\leq C_\ep \Big[\|v_0-\ve(0)\| _{L^1_x}+ \|u_0-\ue(0)\|_{W^{1,1}_x} + \|w_0-\we(0)\|_{W^{1,1}_x} + 1\Big] \ e^{C_\ep t}
	\end{align}
	where the constant $C_\ep$ tends to $+\infty$ as $\ep \rightarrow 0$.	
\end{lemma}

\medskip
\begin{rem}
	The previous estimates \eqref{eq:L^1-uw} and \eqref{eq:L^1-px-uw} are local in time and depend on $\ep$ but in fact, we will never use them in a quantitative fashion. 
	Note that the only point we are interested in is the fact that $u(t, \cdot) - u_\ep(t, \cdot)$ and $v(t, \cdot) - v_\ep(t, \cdot)$ are in $L^1(\R)$ for all $t \geq 0$.
\end{rem}	

\medskip

\begin{proof}
	The proof of this result follows the same lines as before.
	It relies on a combination of $L^1$-estimates for the three following equations
	\begin{align*}
	& \partial_t(v-\ve) = \partial_x(u-\ue) ,\\
	&  \partial_t \partial_x (u -\ue) - \mu\partial_x\left(\frac{1}{v}\partial^2_{x}(u-\ue)\right) \\
	& \quad = - \partial^2_{x}(p_\ep(v) -p_\ep(\ve)) 
	+ \mu \partial^2_{x}\left(\left(\frac{1}{v}-\frac{1}{\ve}\right)\partial_x \ue \right)
	- \mu \partial_x\left(\dfrac{\partial_x v}{v^2} \partial_x(u-\ue)\right) ,\\
	&\partial_t\partial_x(w-\we) = - \partial^2_{x}(p_\ep(v) - p_\ep(\ve)) .
	\end{align*}
	As in the previous proof, the key ingredient is the control of $\partial^2_x v$ in terms of $\partial^k_x\ve,\partial^k_x(u -\ue)$, $\partial^k_x(w-\we)$, $k=0,1$:
	\begin{align*}
	\partial^2_{x} v
	& = \left(\dfrac{\partial_x v }{\ve} - \dfrac{v\partial_x \ve}{\ve^2} \right) \partial_x \ve
	+ \dfrac{v}{\ve}\partial^2_{x}\ve + \dfrac{\partial_x v}{\mu}\big[(u-\ue) -(w-\we)\big] \\
	& \quad + \dfrac{v}{\mu}\big[\partial_x(u-\ue) -\partial_x(w-\we)\big]
	\end{align*}
	and therefore
	\begin{align*}
	\|\partial^2_{x} v\|_{L^1_x}
	& \leq C_\ep \Big(\|\partial_{x} \ve\|_{L^1_x} + \|\partial^2_{x} \ve\|_{L^1_x} 
	+ \|u-\ue\|_{L^1_x} + \|w-\we\|_{L^1_x} \\
	& \qquad + \|\partial_x(u-\ue)\|_{L^1_x} + \|\partial_x(w-\we)\|_{L^1_x} \Big).
	\end{align*}
	Thanks to this bound, we can estimate 
	\begin{align*}
	\partial^2_{x}(p_\ep(v) -p_\ep(\ve))
	& = p'_\ep(v) \partial^2_{x}(v-\ve) + p''_\ep(v)(\partial_xv)^2 - p''_\ep(\ve)(\partial_x\ve)^2 \\
	& \quad	+ (p'_\ep(v) - p'_\ep(\ve))\partial^2_{x}\ve
	\end{align*}
	as follows
	\begin{align*}
	& \|\partial^2_{x}(p_\ep(v) -p_\ep(\ve))\|_{L^1_x} \\
	& \leq C\Big[\ep^{-1/\gamma}\|\partial^2_{x}(v-\ve)\|_{L^1_x} 
	+ \ep^{-2/\gamma}(\|\partial_{x}v\|_{L^1_x} +\|\partial_{x}v_\ep\|_{L^1_x}) 
	+ \ep^{-1/\gamma}\|\partial^2_{x}\ve\|_{L^1_x} \Big]\\
	& \leq C\ep^{-2/\gamma}\Big[ \|\partial_{x} \ve\|_{L^1_x} + \|\partial^2_{x} \ve\|_{L^1_x} 
	+ \|u-\ue\|_{L^1_x} + \|w-\we\|_{L^1_x} \\
	& \hspace{2cm} + \|\partial_x(u-\ue)\|_{L^1_x} + \|\partial_x(w-\we)\|_{L^1_x} \Big].
	\end{align*}

	Furthermore, using  Lemma \ref{lem:stability-u},
	\begin{align*}
&	\left\|  \partial_x\left(\dfrac{\partial_x v}{v^2} \partial_x(u-\ue)\right) \right \|_{L^1((0,t)\times \R_x)}\\
&\leq  C \|\partial_x v\|_{L^2((0,t)\times \R_x)}  \|\partial_x^2 (u-\ue)\|_{L^2((0,t)\times \R_x)}  \\
& \quad + C  \|\partial^2_x v\|_{L^2((0,t)\times \R_x)} \|\partial_x (u-\ue)\|_{L^2((0,t)\times \R_x)} \\
& \quad + C \|\partial_x v\|_{L^2((0,t),L^4(\R_x))}^2 \|\partial_x (u-\ue)\|_{L^\infty((0,t),L^2(\R_x))}  \\
&\leq  C (\| (u-\ue)(0)\|_{H^1} + \delta \ep^{\frac{1}{2\gamma}}) \Bigg[\|\partial_x \ve\|_{L^2((0,t), H^1(\R_x))}  + \|\partial_x^2 V\|_{L^2((0,t), H^1(\R_x))}\\
& \hspace{4cm} +\Big(\|\partial_x \ve\|_{L^2((0,t), L^4(\R_x))} + \|\partial_x^2 V\|_{L^2((0,t), L^4(\R_x))}\Big)^2 \Bigg] \\
&\leq  C_\ep .
	\end{align*}

	Equipped with these estimates we easily deduce~\eqref{eq:L^1-px-uw}.

\end{proof}

\subsection*{Proof of the first part of Theorem \ref{thm:estimates}}
Let us now recap the conclusion of the previous steps. 

Let $(u_0, v_0)$ be an initial data satisfying the assumptions of Theorem \ref{thm:estimates}, and let $U_0, V_0, W_0$ be the associated integrated quantities. Let $(V,W)$ be the solution of \eqref{eq:NL-VW}.
Then according to Lemma \ref{lem:stability-u}, the associated couple $(u,v):=(\partial_x U + u_\ep, \partial_x V + v_\ep)$ is a solution of \eqref{eq:NS-ep} and belongs to $(u_\ep, v_\ep) + \mathcal C([0, \infty), H^1(\R)^2)$, and $v-v_\ep \in L^2(\R_+, H^2)$.
Lemmas \ref{lem:uw-L1} and \ref{lem:v-L1} ensure that for all $t\geq 0$, $(u,v,w)(t)\in (u_\ep, v_\ep, w_\ep) + L^1_0(\R)$.

Conversely, let $(u,v)\in (u_\ep , v_\ep) +  \mathcal C(\R_+, H^1\cap L^1_0(\R))$ be any solution of \eqref{eq:NS-ep} such that $v-v_\ep \in L^2(\R_+, H^2)$, and assume that the initial data $(u_0, v_0)$ satisfies the assumptions of Theorem \ref{thm:estimates}. Define the integrated quantities 
\[
U(t,x):=\int_{-\infty}^{x} (u(t,z) - u_\ep(t,z))dz,\quad V(t,x):=\int_{-\infty}^{x} (v(t,z) - v_\ep(t,z))dz, 
\]
and
\[
W:= U - \mu \frac{\partial_x V}{v_\ep} - \mu \left[\ln \left(1+ \frac{\partial_x V}{v_\ep}\right) - \frac{\partial_x V}{v_\ep}\right].
\]
Then $(V,W)$ is a solution of \eqref{eq:NL-VW}. 
Furthermore, $\partial_x V \in \mathcal C(\R_+, H^1\cap L^1_0)\cap L^2(\R_+, H^2)$ and $\partial_x W\in  \mathcal C(\R_+, H^1\cap L^1_0)$. 
In order to conclude that $(V,W)$ is the unique solution of \eqref{eq:NL-VW} in $B_\delta$, we first need to prove that $(V,W)\in \cX$.  The regularity assumptions on $(u,v)$ ensure that $\partial_t (V,W)\in \mathcal C(\R_+, H^1)$, and therefore $(V,W)\in \mathcal C(\R_+, H^2)$. We infer that $(V,W)\in \cX$. A simple bootstrap argument then ensures that $(V,W)\in B_\delta$, and thus $(V,W)$ is uniquely determined as the fixed point of the application $\mathcal A^\ep$, see Proposition \ref{prop:fixed-point}. The uniqueness of $(u,v)$ follows easily.

As a consequence, we have proved that for any initial data $(u_0, v_0)$ satisfying the assumptions of Theorem \ref{thm:estimates}, there exists a unique solution $(u,v)$ of \eqref{eq:NS-ep} such that
\[
\ba
u-u_\ep \in \mathcal C(\R_+, H^1\cap L^1_0),\\
v-v_\ep  \in \mathcal C(\R_+, H^1\cap L^1_0)\cap L^2(\R_+, H^2).
\ea 
\]

\medskip
\subsection*{Long-time behavior}
We have shown in the previous section that
\[
v-\ve = \partial_x V \in L^2([0,+\infty);H^2(\R)).
\]
Combining this bound with the control of
\[
\partial_t (v-\ve) = \partial_x(u-\ue) \quad \text{in} \quad L^2([0,+\infty);H^1(\R)),
\]
we infer that
\[
\|(v-\ve)(t)\|_{H^1_x} \underset{t\rightarrow +\infty}{\longrightarrow} 0.
\]
As a consequence, we have
\begin{equation}
|(v-\ve)(t,x)| \leq C \|(v -\ve)(t)\|_{L^2_x}^{1/2} \|\partial_x(v -\ve)(t)\|_{L^2_x}^{1/2}  \underset{t\rightarrow +\infty}{\longrightarrow} 0.
\end{equation}
Similarly for $u-\ue$, the bounds obtained in Lemma~\ref{lem:stability-u} yield
\[
\|(u-\ue)(t)\|_{L^2_x} \underset{t\rightarrow +\infty}{\longrightarrow} 0
\]
and therefore
\begin{equation}
|(u-\ue)(t,x)| \leq C \|(u -\ue)(t)\|_{L^2_x}^{1/2} \|\partial_x(u -\ue)\|_{L^\infty L^2_x}^{1/2}  \underset{t\rightarrow +\infty}{\longrightarrow} 0.
\end{equation}

\medskip
%%%%%%%%%%%%%%%%%%%%%%%%%%%%%%%%

\section{Proofs of Lemmas \ref{lem:commutator}, \ref{lem:F-ep_G-ep} and \ref{lem:diff-F-ep_G-ep}}{\label{sec:appendix}}

\subsection{Structure of the commutator.}\label{ssec:commutator}
Let us prove the three properties claimed in Lemma~\ref{lem:commutator}.
A direct calculation gives first
\begin{align*}
[\cLep, \p_x]\begin{pmatrix}f\\g\end{pmatrix}
& = \cLep \begin{pmatrix}\partial_x f\\\partial_x g\end{pmatrix} - \partial_x \left( \cLep \begin{pmatrix} f\\ g\end{pmatrix} \right) \\
& = \begin{pmatrix}p'_\ep(\ve) \partial^2_{x} g \\ -\partial^2_{x}f - \mu\partial_x\left(\frac{\partial^2_{x}g}{\ve}\right) \end{pmatrix}
- \begin{pmatrix}p''_\ep(\ve)\partial_x\ve \partial_{x} g + p'_\ep(\ve) \partial^2_{x} g \\ -\partial^2_{x}f - \mu\partial_x\left(\frac{\partial^2_{x}g}{\ve}\right) + \mu\partial_x\left(\frac{\partial_x \ve}{\ve^2}\partial_{x}g\right) \end{pmatrix} \\
& =  \begin{pmatrix}
- p_\ep''(\ve) \p_x \ve \p_x g\\
-\mu \p_x \left(\frac{\p_x \ve}{\ve^2}\p_x g\right)\end{pmatrix}.
\end{align*}
Next, we have
\begin{align*}
[\cLep,\partial^2_{x}] \begin{pmatrix}f\\g\end{pmatrix}
& = [\cLep,\partial_{x}]\partial_x\begin{pmatrix}f\\g\end{pmatrix} + \partial_x[\cLep,\partial_x]\begin{pmatrix}f\\g\end{pmatrix} \\
& = \begin{pmatrix} -p''_\ep(\ve)\partial_x \ve \partial^2_{x}g \\ - \mu \partial_x \left(\frac{\partial_x \ve}{\ve^2}\partial^2_{x}g\right)\end{pmatrix}
+ \begin{pmatrix} -\partial_x(p''_\ep(\ve)\partial_x \ve \partial_{x}g) \\ - \mu \partial^2_{x} \left(\frac{\partial_x \ve}{\ve^2}\partial_{x}g\right)\end{pmatrix} \\
& = \begin{pmatrix}- 2p_\ep''(\ve) \p_x \ve \partial^2_{x} g\\-2\mu \p_x \left(\frac{\p_x \ve}{\ve^2}\partial^2_{x} g\right)\end{pmatrix}
- \begin{pmatrix} \partial_x( p_\ep''(\ve) \p_x \ve) \p_x g\\ -\mu \p_x \left(\partial_x\left(\frac{\p_x \ve}{\ve^2}\right)\p_x g\right)\end{pmatrix} \\
& = 2 [\cLep, \p_x]\begin{pmatrix}\partial_xf\\\partial_xg\end{pmatrix}
- \begin{pmatrix} \partial_x( p_\ep''(\ve) \p_x \ve) \p_x g\\ -\mu \p_x \left(\partial_x\left(\frac{\p_x \ve}{\ve^2}\right)\p_x g\right)\end{pmatrix}.
\end{align*}
For the third point, 
\begin{align*}
& \int_0^T\int_{\R}{[\cLep, \p_x]\begin{pmatrix}f\\g\end{pmatrix} \cdot \begin{pmatrix}\dfrac{-\partial_xf}{p'_\ep(\ve)} \\ \partial_xg\end{pmatrix}} \\
& = \int_0^T\int_{\R}{\dfrac{p''_\ep(\ve)}{p'_\ep(\ve)}\partial_x \ve \partial_x g\partial_x f}  
- \mu\int_0^T\int_{\R}{\partial_x\left(\dfrac{\partial_x \ve}{\ve^2}\partial_x g\right) \partial_x g}  \\
& = \int_0^T\int_{\R}{\dfrac{p''_\ep(\ve)}{p'_\ep(\ve)}\partial_x \ve \partial_x g\partial_x f}  
- \dfrac{\mu}{2}\int_0^T\int_{\R}{\partial_x\left(\dfrac{\partial_x \ve}{\ve^2}\right) |\partial_x g|^2}  \\
\end{align*}
where the right-hand side can be estimated as follows
\begin{align*}
& \left|\int_0^T\int_{\R}{\dfrac{p''_\ep(\ve)}{p'_\ep(\ve)}\partial_x \ve \partial_x g\partial_x f} \right| 
+ \left|\dfrac{\mu}{2}\int_0^T\int_{\R}{\partial_x\left(\dfrac{\partial_x \ve}{\ve^2}\right) |\partial_x g|^2}\right| \\
& \leq \left\|\frac{p''_\ep(\ve)}{p_\ep'(\ve)} \right\|_\infty \left(\int_0^T\int_{\R}\partial_x \ve |\partial_x f|^2\right)^{1/2} \left(\int_0^T\int_{\R}{|\partial_x g|^2}\right)^{1/2} \\
& \quad + C \left\|\partial_x\left(\dfrac{\partial_x \ve}{\ve^2}\right)\right\|_{L^\infty} \int_0^T\int_{\R}{|\partial_x g|^2} \\
& \leq \delta \int_0^T\int_{\R}\partial_x \ve |\partial_x f|^2 + \frac{C}{\delta }\left( \left\| \frac{1}{\ve-1}\right\|_\infty^2 +\left\|\partial_x\left(\dfrac{\partial_x \ve}{\ve^2}\right)\right\|_{L^\infty}\right) \int_0^T\int_{\R}{|\partial_x g|^2} .
\end{align*}
Using \eqref{est:v-infty} and \eqref{eq:d2v_ep}, we obtain the result announced in Lemma \ref{lem:commutator}.

\subsection{Estimates on the nonlinear terms.}\label{ssec:F-ep_G-ep}

\medskip
\paragraph{Proof of Lemma \ref{lem:F-ep_G-ep}.}
We recall that
\[
F_\ep(f)= - \left[p_\ep (\ve + f) - p_\ep (\ve) - p_\ep'(\ve) f\right],
\]
and that the function $p_\ep$ is $\mathcal C^\infty$ in $]1, + \infty[$. As a consequence, we will extensively use Taylor identities to bound $F_\ep$ and its derivatives. Let us also mention that we will only consider functions $f$ such that $\|f\|_\infty \leq \delta \ep^{1/\gamma}$ for some constant $\delta<1$, so that $|f|\leq \delta (\ve-1)$. As a consequence, for all $k\in \N$ and for all $\delta<1/2$, there exists a constant $C_{k}$ such that
\[
C_{k}^{-1} |p_\ep^{(k)}(\ve)|\leq |p_\ep^{(k)}(\ve + f)|\leq C_{k} |p_\ep^{(k)}(\ve)|.
\]
As a consequence, we infer easily that
\[
|F_\ep(f)|\leq C p_\ep''(\ve) f^2\leq C \frac{p_\ep(\ve)}{(\ve-1)^2 }f^2.
\]
The estimates on $\partial_x^k (F_\ep(f))$ follow from similar arguments after differentiation. We have
\begin{align*}
\partial_x (F_\ep(f))
& = -\partial_x \ve \left[p_\ep' (\ve + f) - p_\ep' (\ve) - p_\ep''(\ve) f\right]\\
& \quad - \partial_x f \left[p_\ep' (\ve + f) - p_\ep' (\ve)\right],
\end{align*}
and therefore
\[
|\partial_x (F_\ep(f))| \leq C \left[\partial_x \ve| p_\ep^{(3)}(\ve)|  f^2 +  p_\ep''(\ve) |f| |\partial_x f|\right].\]
In a similar manner, we have for the second derivative
\begin{align*}
\partial_x^2 (F_\ep(f))
& = -\partial_x^2  \ve \left[p_\ep' (\ve + f) - p_\ep' (\ve) - p_\ep''(\ve) f\right]\\
& \quad - (\partial_x \ve)^2 \left[p_\ep'' (\ve + f) - p_\ep'' (\ve) - p_\ep^{(3)}(\ve) f\right]\\
& \quad - 2\partial_x \ve \partial_x f \left[p_\ep'' (\ve + f) - p_\ep'' (\ve)\right]\\
& \quad - (\partial_x f)^2 p_\ep'' (\ve + f) - \partial_x^2 f \left[p_\ep' (\ve + f) - p_\ep' (\ve)\right].
\end{align*}
As a consequence, using inequalities \eqref{est:v-infty} and  \eqref{eq:d2v_ep}, we obtain
\begin{align*}
| \partial_x^2 (F_\ep(f))| 
& \leq  C \ep^{-1/\gamma} \partial_x \ve |p_\ep^{(3)}(\ve)| f^2\\
& \quad + C \partial_x \ve p_\ep^{(4)}(\ve) f^2\\
& \quad + C \partial_x \ve | p_\ep^{(3)}(\ve)|\; |f|\; |\partial_x f|\\
& \quad + C  p_\ep'' (\ve) (\partial_x f)^2 + C  p_\ep'' (\ve) |f|\; |\partial_x^2 f|.
\end{align*}
Using Young's inequality, we obtain the estimate announced in the lemma. The estimates on $G_\ep$ are similar and are left to the reader.

\paragraph{Proof of Lemma \ref{lem:diff-F-ep_G-ep}.} Once again we focus on $F_\ep$. 
The estimates for $F_\ep(f_1) - F_\ep(f_2)$, $\partial_x (F_\ep(f_1) - F_\ep(f_2))$ go along the same lines as above and are left to the reader. The only novelty in $\partial_x^2 (F_\ep(f_1) - F_\ep(f_2)) $ comes from the term $(\partial_x f_2)^2 p_\ep'' (\ve+ f_2) -  (\partial_x f_1)^2 p_\ep'' (\ve+ f_1)$, for which we write
\begin{align*}
&(\partial_x f_2)^2 p_\ep'' (\ve+ f_2) -  (\partial_x f_1)^2 p_\ep'' (\ve+ f_1) \\
&\quad =(\partial_x f_2)^2 \left[ p_\ep'' (\ve+ f_2)-  p_\ep'' (\ve+ f_1)\right]
+ (\partial_x f_2 - \partial_x f_1) (\partial_x f_2 + \partial_x f_1) p_\ep'' (\ve+ f_1),
\end{align*}
and therefore
\begin{align*}
& \left| (\partial_x f_2)^2 p_\ep'' (\ve+ f_2) -  (\partial_x f_1)^2 p_\ep'' (\ve+ f_1)\right| \\
& \quad \leq C\left ( |p_\ep^{(3)}(\ve) |(\partial_x f_2)^2 |f_1-f_2| + |p_\ep''(\ve) ||\partial_x f_2 - \partial_x f_1|\;  |\partial_x f_2 + \partial_x f_1|\right).
\end{align*}

\section*{Acknowledgements}

This project has received funding from the European Research Council (ERC) under the European Union's Horizon 2020 research and innovation program Grant agreement No. 637653, project BLOC ``Mathematical Study of Boundary Layers in Oceanic Motion''.
C. P. was partially supported by a CNRS PEPS JCJC grant.
This work was  supported by the SingFlows project, grant ANR-18-CE40-0027 of the French National Research Agency (ANR).

%%%%%%%%%%%%%%%%%%%%%%%

\bibliography{biblio}

\begin{thebibliography}{10}

\bibitem{bresch2006}
{\sc Bresch, D., and Desjardins, B.}
\newblock On the construction of approximate solutions for the 2d viscous
  shallow water model and for compressible navier--stokes models.
\newblock {\em Journal de math{\'e}matiques pures et appliqu{\'e}es 86}, 4
  (2006), 362--368.

\bibitem{bresch2003}
{\sc Bresch, D., Desjardins, B., and Lin, C.-K.}
\newblock On some compressible fluid models: {K}orteweg, lubrication, and
  shallow water systems.
\newblock {\em Comm. Partial Differential Equations 28}, 3--4 (2003), 843--868.

\bibitem{bresch2019}
{\sc Bresch, D., Lannes, D., and Metivier, G.}
\newblock Waves interacting with a partially immersed obstacle in the
  boussinesq regime.
\newblock {\em arXiv preprint arXiv:1902.04837\/} (2019).

\bibitem{bresch2014}
{\sc Bresch, D., Perrin, C., and Zatorska, E.}
\newblock Singular limit of a {N}avier--{S}tokes system leading to a
  free/congested zones two-phase model.
\newblock {\em Comptes Rendus Mathematique 352}, 9 (2014), 685--690.

\bibitem{bresch2017}
{\sc Bresch, D., and Renardy, M.}
\newblock Development of congestion in compressible flow with singular
  pressure.
\newblock {\em Asymptotic Analysis 103}, 1-2 (2017), 95--101.

\bibitem{colombo2016}
{\sc Colombo, R., Guerra, G., and Schleper, V.}
\newblock The compressible to incompressible limit of one dimensional euler
  equations: The non smooth case.
\newblock {\em Archive for Rational Mechanics \& Analysis 219}, 2 (2016).

\bibitem{degond2011}
{\sc Degond, P., Hua, J., and Navoret, L.}
\newblock Numerical simulations of the {E}uler system with congestion
  constraint.
\newblock {\em Journal of Computational Physics 230}, 22 (2011), 8057--8088.

\bibitem{denisova2018}
{\sc Denisova, I., and Solonnikov, V.}
\newblock Local and global solvability of free boundary problems for the
  compressible navier--stokes equations near equilibria.
\newblock {\em Handbook of Mathematical Analysis in Mechanics of Viscous
  Fluids\/} (2018), 1--88.

\bibitem{guerra2016}
{\sc Guerra, G., and Schleper, V.}
\newblock A coupling between a 1d compressible-incompressible limit and the 1d
  p-system in the non smooth case.
\newblock {\em Bulletin of the Brazilian Mathematical Society, New Series 47},
  1 (2016), 381--396.

\bibitem{haspot2014}
{\sc Haspot, B.}
\newblock Existence of global strong solution for the compressible
  {N}avier-{S}tokes equations with degenerate viscosity coefficients in 1d.
\newblock {\em arXiv preprint arXiv:1411.5503\/} (2014).

\bibitem{haspot2018}
{\sc Haspot, B.}
\newblock {Vortex solutions for the compressible {N}avier-{S}tokes equations
  with general viscosity coefficients in 1D: regularizing effects or not on the
  density}.
\newblock {\em preprint HAL hal-01716150\/} (2018).

\bibitem{humpherys2010}
{\sc Humpherys, J., Lafitte, O., and Zumbrun, K.}
\newblock Stability of isentropic {N}avier--{S}tokes shocks in the high {M}ach
  number limit.
\newblock {\em Communications in Mathematical Physics 293}, 1 (2010), 1--36.

\bibitem{iguchi2018}
{\sc Iguchi, T., and Lannes, D.}
\newblock Hyperbolic free boundary problems and applications to wave-structure
  interactions.
\newblock {\em arXiv preprint arXiv:1806.07704\/} (2018).

\bibitem{liu2009}
{\sc Liu, T.-P., and Zeng, Y.}
\newblock Time-asymptotic behavior of wave propagation around a viscous shock
  profile.
\newblock {\em Communications in Mathematical Physics 290}, 1 (2009), 23--82.

\bibitem{mascia2004}
{\sc Mascia, C., and Zumbrun, K.}
\newblock Stability of large-amplitude viscous shock profiles of
  hyperbolic-parabolic systems.
\newblock {\em Archive for rational mechanics and analysis 172}, 1 (2004),
  93--131.

\bibitem{matsumura1985}
{\sc Matsumura, A., and Nishihara, K.}
\newblock On the stability of travelling wave solutions of a one-dimensional
  model system for compressible viscous gas.
\newblock {\em Japan Journal of Applied Mathematics 2}, 1 (1985), 17.

\bibitem{matsumura2010}
{\sc Matsumura, A., and Wang, Y.}
\newblock Asymptotic stability of viscous shock wave for a onedimensional
  isentropic model of viscous gas with density dependent viscosity.
\newblock {\em Methods and Applications of Analysis 17}, 3 (2010), 279--290.

\bibitem{mellet2008}
{\sc Mellet, A., and Vasseur, A.}
\newblock Existence and uniqueness of global strong solutions for
  one-dimensional compressible navier--stokes equations.
\newblock {\em SIAM Journal on Mathematical Analysis 39}, 4 (2008), 1344--1365.

\bibitem{perrin2015}
{\sc Perrin, C., and Zatorska, E.}
\newblock Free/congested two-phase model from weak solutions to
  multi-dimensional compressible {N}avier-{S}tokes equations.
\newblock {\em Communications in Partial Differential Equations 40}, 8 (2015),
  1558--1589.

\bibitem{serre1999}
{\sc Serre, D.}
\newblock {\em Systems of Conservation Laws 1: Hyperbolicity, entropies, shock
  waves}.
\newblock Cambridge University Press, 1999.

\bibitem{shelukhin1984}
{\sc Shelukhin, V.}
\newblock On the structure of generalized solutions of the one-dimensional
  equations of a polytropic viscous gas.
\newblock {\em Journal of Applied Mathematics and Mechanics 48}, 6 (1984),
  665--672.

\bibitem{shibata2016}
{\sc Shibata, Y.}
\newblock On the $\mathcal{R}$-boundedness for the two phase problem with phase
  transition: Compressible-incompressible model problem.
\newblock {\em Funkcialaj Ekvacioj 59}, 2 (2016), 243--287.

\bibitem{vasseur2016}
{\sc Vasseur, A.~F., and Yao, L.}
\newblock Nonlinear stability of viscous shock wave to one-dimensional
  compressible isentropic {N}avier-{S}tokes equations with density dependent
  viscous coefficient.
\newblock {\em Commun. Math. Sci 14}, 8 (2016), 2215--2228.

\end{thebibliography}

          \end{document}